\documentclass[12pt,a4paper, oneside,reqno]{article}

\usepackage[utf8]{inputenc}
\usepackage[english]{babel}

\usepackage[left=1in,right=1in,top=1in,bottom=1in]{geometry}

\usepackage{amsfonts,amssymb,amsmath,amsthm}

\usepackage[symbol]{footmisc} 
\usepackage{longtable}
\usepackage{tasks}

\usepackage{enumerate}
\usepackage{cite}
\usepackage{hyperref}
\usepackage{tkz-graph}

\usepackage{fancyhdr} 
\newcommand{\black}{\color{black}}

\theoremstyle{plain}
\newtheorem{theorem}{Theorem}
\newtheorem{lemma}[theorem]{Lemma}
\newtheorem{proposition}[theorem]{Proposition}
\newtheorem{remark}[theorem]{Remark}

\newtheorem{definition}[theorem]{Definition}
\newtheorem{corollary}[theorem]{Corollary}

\usetikzlibrary{arrows}

\newtheorem*{theoremA1}{Theorem A1}
\newtheorem*{theoremA2}{Theorem A2}

\newtheorem*{theoremG1}{Theorem G1}
\newtheorem*{theoremG2}{Theorem G2}

\usepackage{fouriernc} 

\begin{document}


\bigskip

\noindent{\Large
The algebraic and geometric classification of \\
  commutative post-Lie algebras}
 \footnote{
The    work is supported by 
FCT   2023.08031.CEECIND, 2023.08952.CEECIND, and
UID/00212/2025.}

 \bigskip

\begin{center}

 {\bf
Hani Abdelwahab\footnote{Department of Mathematics, 
 Mansoura University,  Mansoura, Egypt; \ haniamar1985@gmail.com}, 
 Kobiljon Abdurasulov\footnote{CMA-UBI, University of  Beira Interior, Covilh\~{a}, Portugal;  \ 
Romanovsky Institute of Mathematics, Academy of Sciences of Uzbekistan, Tashkent, Uzbekistan; \ abdurasulov0505@mail.ru} \&      
Ivan Kaygorodov\footnote{CMA-UBI, University of  Beira Interior, Covilh\~{a}, Portugal; \    kaygorodov.ivan@gmail.com}  }

\end{center}

\noindent {\bf Abstract:}
{\it  
We study commutative post-Lie algebras $(${\rm CPA}s$)$ from an algebraic point of view.
Firstly, we find some new identities in {\rm CPA}, 
which shows that the commutative multiplication gives a medial and derived commutative associative algebra. 
As corollaries, we have that there are no simple nontrivial commutative post-Lie algebras and that perfect Lie and centrless perfect commutative associative algebras do not admit nontrivial {\rm CPA} structures. The identities of depolarized {\rm CPA}s are defined. Based on the obtained identities, we developed a method for the classification of $n$-dimensional {\rm CPA}s and gave the algebraic classification of $3$-dimensional {\rm CPA}. We also developed another method for classifying $n$-dimensional nilpotent {\rm CPA}s from nilpotent {\rm CPA}s of smaller dimension and gave the algebraic classification of $4$-dimensional nilpotent {\rm CPA}s.
Based on the obtained results, we present the geometric classifications of complex 
$3$-dimensional and $4$-dimensional nilpotent {\rm CPA}s.}

 \medskip 


\medskip


	 \medskip

 
\tableofcontents


\section*{Introduction}

The idea of considering algebras with two or more multiplications is not new, but it is under a certain consideration \cite{DET,Khr,Khr26,afm,aae23,MF,kms}.
The most well-known example of algebras with two multiplications is Poisson algebras. 
 The post-Lie operad was defined by  Vallette \cite{VB} as the Koszul dual of the commutative triassociative operad; he also proved that both these operads are Koszul. 
 To be precise, a post-Lie algebra is a vector space ${\rm P}$  with two bilinear operations $\circ$ and $[\cdot,\cdot]$ such that $({\rm P}, [\cdot,\cdot])$ is a Lie algebra and the two operations satisfy the following relations:
 \begin{longtable}{rcl}
$(x\circ y) \circ z - x\circ (y \circ z) - (x\circ z)\circ y+x\circ (z\circ y)$&$ =$&$ x\circ [y,z],$\\ 
$[x,y] \circ z $&$ =$&$ [x\circ z,y]+[x,y\circ z].$
\end{longtable}\noindent
The first of these relations implies that if  $[x,y]=0$ then $({\rm P}, \circ)$ is a pre-Lie algebra (also called a left-symmetric algebra). The second says that right multiplications by $\circ$ are derivations of $({\rm P}, [\cdot,\cdot]).$
   A post-Lie algebra $({\rm P}, \circ , [\cdot,\cdot])$ has another associated Lie algebra structures 
   given by the multiplication  $\{x,y\}= x\circ y - y\circ x+[x,y].$
Burde,   Dekimpe, and  Vercammen  defined the notion of post-Lie structures on the pair of two Lie multiplications $([\cdot,\cdot], \{\cdot, \cdot\})$ on the same vector space in \cite{BDV12}.
Namely, a multiplication $\cdot$ is called a post-Lie structure on the pair $([\cdot,\cdot], \{\cdot, \cdot\})$ of Lie multiplications   if it satisfies
\begin{longtable}{rcl}
$x \cdot y - y \cdot x$&$=$&$[x,y]-\{x,y\},$\\
$[x,y] \cdot z$&$=$&$x\cdot (y\cdot z)- y\cdot (x\cdot z),$\\
$x\cdot\{y,z\}$&$=$&$\{x\cdot y,z\}+\{y,x\cdot z\}.$
\end{longtable}

Post-Lie algebras arise naturally from the differential geometry of homogeneous
spaces and Klein geometries, topics that are closely related to Cartan’s method of moving frames.
They are related to moving frame theory \cite{ML13,B-sur}. It was also found that post-Lie algebras play an essential role in regularity structures in stochastic analysis \cite{br1,br2,H}.
They also appeared in differential geometry  \cite{AEM2}, and numerical analysis \cite{CEO}.
Recently, post-Lie algebras have been studied from different points of view, including constructions of nonabelian generalized Lax pairs \cite{6}, 
Poincar\'e--Birkhoff--Witt type theorems \cite{18}, 
factorization theorems \cite{20}, 
and relations to post-Lie groups \cite{AEM,MQS,BGST}.
A cohomology theory corresponding to operator homotopy
post-Lie algebras was also given in \cite{50} and   the homotopy theory of post-Lie algebras
is developed in full generality, including the notion of post-Lie$_\infty$ algebra and the concomitant
cohomology theory in \cite{51}. 
A bialgebra theory of post-Lie algebras was established in \cite{LBG}.

If $\circ$ is commutative then the post-Lie algebra $({\rm P}, \circ, [\cdot,\cdot])$ gives a commutative post-Lie algebra.
It looks like the algebraic study of commutative post-Lie algebras (CPA for short) starts in \cite{BD16}, 
where Burde and Dekimpe proved that any CPA structure on a complex finite-dimensional semisimple Lie algebra is trivial and gave some examples of $3$-dimensional commutative post-Lie algebras.
Later, Burde and   Moens  proved that any complex finite-dimensional Lie algebra admitting a nondegenerate CPA structure is solvable and, conversely, it is shown that any nontrivial solvable Lie algebra admits a nontrivial CPA structure \cite{BM16}.
CPA structures on stem Lie algebras and free nilpotent Lie algebras were studied in \cite{BDM19};
on modular Lie algebras were studied in \cite{BM20};
on Kac-Moody algebras were described by Burde and Zusmanovich in \cite{BZ}.
Burde and Ender proved that CPA structures on some nilpotent Lie algebras are Poisson-admissible \cite{BE}.
CPA structures are closely related to biderivations of Lie algebras \cite{TX,CYZ}.
CPAs give a non-associative version of transposed $1$-Poisson algebras, firstly considered in \cite{dP}.
   \medskip

The algebraic classification (up to isomorphism) of algebras of dimension $n$ of a certain variety
defined by a family of polynomial identities is a classic problem in the theory of non-associative algebras, see \cite{ afm,roman2, abcf2 }.
There are many results related to the algebraic classification of small-dimensional algebras in different varieties of
associative and non-associative algebras.
For example, algebraic classifications of 
$2$-dimensional algebras,
$3$-dimensional Poisson algebras,
$4$-dimensional alternative  algebras,
$4$-dimensional nilpotent Poisson  algebras,
$4$-dimensional nilpotent right alternative  algebras,
$5$-dimensional symmetric Leibniz algebras  and so on
 have been given.
 Deformations and geometric properties of a variety of algebras defined by a family of polynomial identities have been an object of study since the 1970's, see \cite{ben,FKS,FKS25,BC99,GRH, akt, roman2 }  and references in \cite{k23,l24,MS}. 
 Burde and Steinhoff constructed the graphs of degenerations for the varieties of    $3$-dimensional and $4$-dimensional Lie algebras~\cite{BC99}. 
 Grunewald and O'Halloran studied the degenerations for the variety of $5$-dimensional nilpotent Lie algebras~\cite{GRH}. 

 \medskip

We continue the  study of commutative post-Lie algebras from an algebraic point of view.
Firstly, in Section \ref{sec1} we find some new identities in {\rm CPA}s, which show that the commutative multiplication gives a medial algebra (Proposition \ref{dercommassoc}).
Now, it is easy to see that each {\rm CPA} satisfies 
$(x \cdot y)^2 = x^2 \cdot y^2$ \cite{T24}.
The medial magma identity itself and its variations are very well-known in the context of quasigroups,
finite geometries, and Steiner designs. The classical result of  Bruck \cite{B44} and Murdoch \cite{M41} characterizes a medial quasigroup as an isotope of a certain abelian group.
On the other hand, to the best of our knowledge, the algebras satisfying the medial magma identity have not yet been systematically studied (see, however, some related train algebras satisfying polynomial identities studied in \cite{LR}).
Next, we proved that there are no simple nontrivial commutative post-Lie algebras (Corollary \ref{corsimple}) and that each  perfect Lie (or centrless perfect commutative associative) algebra does not admit nontrivial {\rm CPA} structures (Corollary \ref{cor}), which extends recently obtained results in papers by Burde, Moens, and Zusmanovich, among others \cite{BZ,BM16}.
The identities of depolarized {\rm CPA}s are defined in Theorem \ref{dep}.
Namely, we proved that commutative post-Lie algebras in one multiplication give a subvariety in 
$\frac 12$-left symmetric algebras, which were previously defined in \cite{K25}.
Based on the obtained identities, we developed a method for the classification of $n$-dimensional {\rm CPA}s and gave the algebraic classification of $3$-dimensional {\rm CPA}:

\medskip 

\noindent 
{\bf Theorem A1} gives the description of nontrivial complex $3$-dimensional {\rm CPA}s in terms of
 $9$ non-isomorphic algebras and $7$ one-parameter families of algebras. 

\medskip 

\noindent
We also developed another method for classifying $n$-dimensional nilpotent {\rm CPA}s from nilpotent {\rm CPA}s of smaller dimension and gave the algebraic classification of $4$-dimensional nilpotent {\rm CPA}s.

\medskip 
\noindent 
{\bf Theorem A2} gives the description of nontrivial complex $4$-dimensional nilpotent {\rm CPA}s in terms of
 $20$ non-isomorphic algebras and $7$ one-parameter families of algebras. 

\medskip 
\noindent
Based on the obtained results, we present the geometric classifications of complex 
$3$-dimensional and $4$-dimensional nilpotent {\rm CPA}s:
the variety of $3$-dimensional nilpotent commutative post-Lie
algebras is irreducible and it  has  dimension  $7$ (Proposition \ref{3gnil}).

\medskip 
\noindent 
{\bf Theorem G1} states that the variety of $3$-dimensional {\rm CPA}s is $9$-dimensional with    
$8$  irreducible components and  with $5$ rigid algebras. 

\medskip 
\noindent 
{\bf Theorem G2} states that the variety of $4$-dimensional nilpotent {\rm CPA}s is $14$-dimensional with  $2$  irreducible components and without rigid algebras.


\section{Commutative post-Lie algebras}\label{sec1}

All the algebras below will be over $\mathbb C$ and all the linear maps will be $\mathbb C$-linear.
For simplicity, every time we write the multiplication table of an algebra, the products of basic elements whose values are zero or can be recovered from the commutativity  or the anticommutativity are omitted.
The notion of a nontrivial algebra with two multiplications means that both multiplications are  nonzero.

\begin{definition}
A commutative post-Lie algebra {\rm (CPA)}  is a  vector space ${\rm P}$ equipped with 
a commutative multiplication $"\cdot"$ and 
a Lie algebra multiplication  $"\left\{\cdot , \cdot \right\}",$
which satisfying  
\begin{equation}\label{id1}
\left\{ x,y\right\} \cdot z \ =  x\cdot \left( y\cdot z\right) -y\cdot \left(
x\cdot z\right),\\
\end{equation}
\begin{equation}\label{id2}
x\cdot \left\{ y,z\right\} \  =\ \left\{ x\cdot y,z\right\} +\left\{ y,x\cdot
z\right\}.
\end{equation}%
\end{definition}

\begin{proposition}
    A {\rm CPA} with one trivial multiplication is
    isomorphic to a commutative associative algebra or a Lie algebra.
    
\end{proposition}

\begin{proposition} \label{dercommassoc}
Let $\left( {\rm P},\cdot,\left\{ \cdot,\cdot\right\} \right)$ 
be a commutative post-Lie algebra. Then\footnote{It has to be mentioned that via some calculations in the Albert program, it is possible to find one more interesting identity in {\rm CPA}s:
$\big\{ a\cdot b,\left\{ c,\left\{ d,e\right\} \right\} \big\}  =0,$
which will be omitted in our Proposition due to a very huge proof.
}%

\begin{longtable}{rcl}
$\left( x\cdot y\right) \cdot \left( z\cdot w\right)  $&$=$&$\left( x\cdot
z\right) \cdot \left( y\cdot w\right)$\footnote{The present identity defines the variety of medial algebras considered in \cite{T24}.},\\
$\big( \left( x\cdot y\right) \cdot z\big) \cdot w $&$=$&$\big( \left( x\cdot
y\right) \cdot w\big) \cdot z,$ \\
$ (x \cdot y) \cdot \left\{ z,w\right\} $&$ = $&$0$.\\ 
\end{longtable}
\noindent 
In particular, $\left( {\rm P}, \cdot \right) $ is a medial  and derived commutative associative algebra\footnote{It means that ${\rm A}^2$ is a commutative associative algebra.}.
\end{proposition}

\begin{proof}
Since $\left( {\rm P},\cdot,\left\{ \cdot,\cdot\right\} \right) $ is
a CPA, the following polynomial identities hold:
\begin{longtable}{lclcl}
$f\left( a,b,c\right)  $&$:=$&$\left\{ a,b\right\} \cdot c-a\cdot \left( b\cdot
c\right) +b\cdot \left( a\cdot c\right) $&$=$&$0,$ \\
$g\left( a,b,c\right)  $&$:=$&$a\cdot \left\{ b,c\right\} -\left\{ a\cdot
b,c\right\} -\left\{ b,a\cdot c\right\} $&$=$&$0.$
\end{longtable}
\noindent 
Now, let $x,y,z,w\in {\rm P}$. Then
\begin{longtable}{l}
$4\left\{ w,z\right\} \cdot (x \cdot y) \ = \ 
4\left( \left( x\cdot y\right) \cdot z\right)\cdot w-4\left( \left( x\cdot y\right)\cdot w\right)\cdot z \ = $\\ 
\quad $-f\left( z,x,y\right) \cdot w-f\left( z,y,x\right) \cdot w+f\left(w,x,y\right) \cdot z+f\left( w,y,x\right) \cdot z+f\left( w,z,x\right) \cdot y+f\left( w,z,y\right) \cdot x$\\
\quad $-2f\left( y\cdot x,z,w\right) +2f\left( y\cdot x,w,z\right) +f\left(
z\cdot y,x,w\right) +f\left( z\cdot x,y,w\right)+f\left( z\cdot
y,w,x\right) $ \\
\quad $+f\left( z\cdot x,w,y\right) -f\left( w\cdot z,x,y\right) -f\left( w\cdot x,y,z\right) -f\left( w\cdot
y,z,x\right)-f\left( w\cdot x,z,y\right) +g\left( x,z,y\right) \cdot
w$  \\
\quad $-g\left( x,w,y\right) \cdot z-g\left( x,w,z\right) \cdot y+g\left( y,z,x\right) \cdot w -g\left(y,w,x\right) \cdot z-g\left( y,w,z\right) \cdot x \ =\ 0,$
\end{longtable}
and 
\begin{longtable}{l}
$4\left( x\cdot y\right) \cdot \left( z\cdot w\right) -4\left( x\cdot
z\right) \cdot \left( y\cdot w\right)  \ = $\\
\quad $-f\left( z,y,x\right) \cdot w-f\left( w,x,y\right) \cdot z+f\left( w,x,z\right) \cdot y+f\left(
w,y,z\right) \cdot x-f\left( w,z,y\right) \cdot x$\\
\quad $-f\left( y\cdot x,z,w\right) -f\left( y\cdot x,w,z\right) +f\left( z\cdot
x,y,w\right) +f\left( z\cdot x,w,y\right) -f\left( w\cdot z,x,y\right) $\\
\quad $+f\left( w\cdot y,x,z\right) -f\left( w\cdot z,y,x\right)  +f\left( w\cdot
y,z,x\right) +g\left( x,z,y\right) \cdot w+g\left( y,w,x\right) \cdot z $\\
\quad $ +g\left( y,w,z\right) \cdot x-g\left( z,w,x\right) \cdot y-g\left(
z,w,y\right) \cdot x \ =\ 0.$
\\

\end{longtable}
\end{proof}

\begin{corollary}[\footnote{The finite-dimensional case was proved in  \cite[Theorem 3.4]{BM16}.}]\label{coro4}
Let $({\rm P}, \cdot, \{ \cdot,\cdot\})$ be a {\rm CPA} 
and $\mathfrak{p}$ be a perfect subalgebra in $({\rm P},   \{ \cdot,\cdot\}),$
then ${\mathfrak p} \cdot {\rm P}=0.$
\end{corollary}
\begin{proof}
    
Since $\left\{ \mathfrak{p},\mathfrak{p}\right\} =\mathfrak{p}$, due to Proposition \ref{dercommassoc}, we have $\left\{ \mathfrak{p},\mathfrak{p}\right\} \cdot \left(
\mathfrak{p} \cdot {\rm P}\right) = \mathfrak{p}\cdot \left( \mathfrak{p}\cdot {\rm P}\right) =0$. 
Moreover, we have
\begin{center}
$\mathfrak{p} \cdot {\rm P} \ = \ 
\{ \mathfrak{p} , \mathfrak{p}\} \cdot {\rm P} \ \subseteq  \ \mathfrak{p} \cdot \left( \mathfrak{p} \cdot {\rm P} \right) \ =\ 0.$
\end{center}

\end{proof}

\begin{corollary}\label{corsimple}
There are no simple nontrivial commutative post-Lie algebras.
\end{corollary}
\begin{proof}
    Let $({\rm P}, \cdot, \{\cdot,\cdot\})$ be a nontrivial simple {\rm CPA},
    due to the defined identities, $\{ {\rm P},{\rm P}\}$ is an ideal of  $({\rm P}, \cdot, \{\cdot,\cdot\}).$ Hence,    $\{ {\rm P},{\rm P}\}= {\rm P}.$
   Due to Corollary  \ref{coro4},  we have $ 
     {\rm P} \cdot {\rm P}\  =\   0.$
 
\end{proof}

\begin{lemma}[\footnote{We suppose that it is a well-known results, but we are not available to find it in the literature.}]\label{unital}
Let ${\rm A }$ be a finite-dimensional   perfect commutative associative algebra. Then ${\rm A }$ is unital.
\end{lemma}

\begin{proof}
We proceed by induction on the dimension of ${\rm A }$. 
If ${\rm dim} \ {\rm A} = 1,$ then ${\rm A} = {\rm A}^2$  and it is unital.

Now, assume the result holds for all commutative associative algebras of dimension less than \( n \), and let \( \dim \ {\rm A} = n \geq 2 \). Since \( {\rm A} = {\rm A}^2 \neq 0 \), the algebra is not nilpotent. Therefore, \( {\rm A} \) contains a nonzero idempotent element \( e \in {\rm A} \), that is, \( e^2 = e \) and \( e \neq 0 \).

Define the Peirce components with respect to \( e \) as:
\[
{\rm A}_1 := e{\rm A} = \{ x \in {\rm A} \mid ex = x \}, \quad {\rm A}_0 := \{ x \in {\rm A} \mid ex = 0 \}.
\]
Then \( {\rm A} = {\rm A}_1 \oplus {\rm A}_0 \) as a vector space, and both \( {\rm A}_1 \) and \( {\rm A}_0 \) are two-sided ideals of \( {\rm A} \). Moreover, the multiplication structure satisfies:
\[
{\rm A}^2\  =\  ({\rm A}_1 \oplus {\rm A}_0)^2 \ =\  {\rm A}_1^2 + {\rm A}_1 {\rm A}_0 + {\rm A}_0 {\rm A}_1 + {\rm A}_0^2 \ =\ {\rm A}_1 + {\rm A}_0^2,
\]
since \( {\rm A}_1 {\rm A}_0 \ =\  {\rm A}_0 {\rm A}_1\  =\  0 \) due to the Peirce relations.
Hence, using \( {\rm A} = {\rm A}^2 \), we obtain:
\[
{\rm A}\  =\  {\rm A}_1 \oplus {\rm A}_0\  =\  {\rm A}_1 + {\rm A}_0^2.
\]
By comparing the two decompositions, we conclude \( {\rm A}_0 = {\rm A}_0^2 \), so \( {\rm A}_0 \) also satisfies the same assumption, with \( \dim {\rm A}_0 < \dim {\rm A} \). By the induction hypothesis, \( {\rm A}_0 \) has a multiplicative identity \( f \in {\rm A}_0 \).

We claim that \( 1_{\rm A} := e + f \in {\rm A} \) is a multiplicative identity for all of \( {\rm A} \). Let \( x \in {\rm A} \), and write \( x = x_1 + x_0 \), where \( x_1 \in {\rm A}_1 \) and \( x_0 \in {\rm A}_0 \). Then:
\begin{longtable}{lclcl}
$1_{\rm A} x$ &$=$&$ (e + f)(x_1 + x_0)$&$=$&$ ex_1 + ex_0 + fx_1 + fx_0.$
\end{longtable}\noindent 
Since \( ex_1 = x_1 \) and \( fx_0 = x_0 \), while \( ex_0 = 0 \) and \( fx_1 = 0 \), we conclude:
\begin{center}
    $1_{\rm A} x \ = \ x_1 + x_0 \ = \ x$
\end{center}
and  due to commutativity, \(  {\rm A} \) is unital.
\end{proof}
\black

\begin{corollary}\label{cor}
If $\left({\rm P},\cdot, \{ \cdot,\cdot \} \right)$ is a {\rm CPA}
and $\left({\rm P}, \{ \cdot,\cdot \} \right)$ is   perfect 
$\big($resp., $\left({\rm P},\cdot  \right)$ is  finite-dimensional perfect   or infinite-dimensional centerless  perfect$\big)$, then 
 $\left({\rm P},\cdot  \right)$ $\big($resp., $\left({\rm P},\{ \cdot,\cdot \}   \right)$$\big)$ is   trivial.

 \black
\end{corollary}

\begin{proof}
First,  if       $\{ {\rm P},{\rm P}\}= {\rm P},$ then due to   Corollary \ref{coro4}, we have   $ {\rm P} \cdot {\rm P}\  =0.$

Second, if $ {\rm P} \cdot {\rm P}\ = {\rm P},$ then by Proposition \ref{dercommassoc}, we have   
${\rm P} \cdot \{{\rm P} , {\rm P}\}  =0.$ 
So, in the finite-dimensional case, due to Lemma \ref{unital}, $({\rm P}, \cdot)$ is unital, hence
$\{{\rm P} , {\rm P}\}  =0.$ 
In the infite-dimensional case, due to centerless of $({\rm P}, \cdot),$ we have $\{{\rm P} , {\rm P}\}  =0.$ 

\end{proof}

\begin{theorem}\label{dep}
Consider an algebra $\left( {\rm P},\cdot \right) $ whose underlying
vector space ${\rm P}$ is endowed with   two new products 
$x \circ y := \frac 12 ( x\cdot y+y\cdot x) $ and 
$\left[ x ,y \right]: = \frac 12 ( x\cdot y-y\cdot x).$ 
Then $({\rm P},\circ ,\left[ \cdot
,\cdot \right] )$ is a commutative post-Lie algebra if and only if the 
algebra $\left( {\rm P},\cdot \right) $ satisfies the following
identities:%
\begin{longtable}{lclcl}
${\rm A}\left( x,y,z\right)$&$:=$&$(x\cdot y)\cdot z-x\cdot (z\cdot y)+z\cdot (x\cdot y)-(z\cdot y)\cdot x $&$=$&$0,$ \\
${\rm B}\left( x,y,z\right)$\footnote{The identity ${\rm B}(x,y,z)=0$ is equivalent to $(x,y,z)_{\frac 12} = (y,x,z)_{\frac 12},$ 
where $(x,y,z)_{\frac 12}=\frac 12 (xy)z-x(yz).$
The last identity  defines $\frac 12$-left symmetric algebras introduced in \cite{K25}.}&$:=$&$(y\cdot z)\cdot x-2y\cdot (z\cdot x)+2z\cdot (y\cdot x)-(z\cdot y)\cdot x $&$=$&$0.$
\end{longtable}
\end{theorem}
\begin{proof}
The Jacobi identity and identities \eqref{id1},\eqref{id2} 
are equivalent, respectively, to
\begin{longtable}{lcl}
$h_{1}\left( x,y,z\right)  $&$:=$&$x\cdot (y\cdot z)-x\cdot (z\cdot y)-y\cdot
(x\cdot z)+y\cdot (z\cdot x)+z\cdot (x\cdot y)-z\cdot (y\cdot x)$ \\
&&$-(x\cdot y)\cdot z+(x\cdot z)\cdot y+(y\cdot x)\cdot z-(y\cdot z)\cdot
x-(z\cdot x)\cdot y+(z\cdot y)\cdot x,$\\
$h_{2}\left( x,y,z\right)  $&$:=$&$-x\cdot (y\cdot z)-x\cdot (z\cdot y)+y\cdot
(x\cdot z)+y\cdot (z\cdot x)+z\cdot (x\cdot y)-z\cdot (y\cdot x)$ \\
&&$+(x\cdot y)\cdot z+(x\cdot z)\cdot y-(y\cdot x)\cdot z-(y\cdot z)\cdot
x+(z\cdot x)\cdot y-(z\cdot y)\cdot x,$\\
$h_{3}\left( x,y,z\right)  $&$:=$&$x\cdot (y\cdot z)-x\cdot (z\cdot y)-y\cdot
(x\cdot z)-y\cdot (z\cdot x)+z\cdot (x\cdot y)+z\cdot (y\cdot x)$ \\
&&$-(x\cdot y)\cdot z+(x\cdot z)\cdot y-(y\cdot x)\cdot z+(y\cdot z)\cdot
x+(z\cdot x)\cdot y-(z\cdot y)\cdot x.$
\end{longtable}\noindent
Then we have 
\begin{longtable}{lcl}
${\rm A}\left( x,y,z\right)  $&$=$&$-\frac{1}{2} \big(h_{2}(y,x,z)+h_{2}(z,y,x) \big),$ \\
${\rm B}\left( x,y,z\right)  $&$=$&$\frac 12 \big(h_{1}(z,y,x)-h_{2}(z,y,x)-h_{3}(y,z,x)+h_{3}(z,y,x) \big),$\\
$h_{1}\left( x,y,z\right)  $&$=$&$\frac 13 \big(
2{\rm B}\left( x,z,y\right) -2{\rm B}\left( y,z,x\right)
+2{\rm B}\left( z,y,x\right)-{\rm A}\left( y,z,x\right) -{\rm A}\left( z,x,y\right)
+{\rm A}\left( z,y,x\right) \big),$ \\
$h_{2}\left( x,y,z\right)  $&$=$&$  {\rm A}\left( z,x,y\right)-{\rm A}\left( y,z,x\right) 
-{\rm A}\left( z,y,x\right),$ \\
$h_{3}\left( x,y,z\right)  $&$=$&$ \frac 13 \big( 
{\rm A}\left( z,y,x\right)- {\rm A}\left( y,z,x\right) +5 {\rm A}\left( z,x,y\right)
 -4 {\rm B}(x,z,y)-2{\rm B}(y,z,x)+2{\rm B}(z,y,x) \big).$
\end{longtable}

Identities ${\rm A}$ and ${\rm B}$ are independent due to the following observation:
algebra $Q_{\rm IJ},$ given below satisfies the identity ${\rm I}$ and does not satisfy the identity ${\rm J}:$
\begin{longtable}{lclllll}
$Q_{\rm AB}$ & $:$&$e_1\cdot e_1=e_1$& $e_1\cdot e_2= e_2$ & $e_2\cdot e_1=-e_2$\\
$Q_{\rm BA}$ & $:$&$e_1\cdot e_1=e_1$& $e_1\cdot  e_2=2e_2$ \\
    \end{longtable}
 
\end{proof}

\section{The algebraic classification of CPAs}

\subsection{{\rm CPA}s}

\subsubsection {The method of classification of CPAs with a fixed $"\cdot"$  multiplication}

\begin{definition}
Let $\left( {\rm A},\cdot \right) $ be a derived commutative associative algebra. Define ${\rm Z}^{2}\left( {\rm A},{\rm A}\right) $ to be a set of all antisymmetric bilinear maps 
$\theta :{\rm A}\times {\rm A}\rightarrow {\rm A}$ such that:%

\begin{longtable}{rcl}
$\theta ( x,y) \cdot z $&$=$&$x\cdot \left( y\cdot z\right) -y\cdot \left(
x\cdot z\right),$ \\
$x\cdot \theta ( y,z) $&$=$&$\theta ( x\cdot y,z)  +\theta( y,x\cdot z),$\\
$\theta (x, \theta ( y,z)) $&$=$&$\theta ( \theta (x, y),z)  +\theta( y,\theta (x, z)).$
\end{longtable}
\end{definition}

Now, for $\theta \in {\rm Z}^{2}\left( {\rm A},{\rm A}\right)$, let us define a multiplication 
$\{\cdot,\cdot\}_{\theta }$ on ${\rm A}$ by $\{x,y\}_{\theta} =\theta \left( x,y\right) $ for all $x,y$ in ${\rm A}$. 
Then $\left( {\rm A},\cdot , \{\cdot,\cdot\}_{\theta }\right) $ is a  CPA. 
Conversely, if $\left( {\rm A},\cdot ,\{\cdot,\cdot\} \right) $ is a CPA, 
then there exists $\theta \in {\rm Z}^{2}\left(  {\rm A},{\rm A}\right) $ such that $\left( 
{\rm A},\cdot , \{\cdot,\cdot\}_{\theta }\right) \cong\left( {\rm A},\cdot ,\{\cdot,\cdot\}
\right)$. To see this, consider the bilinear map $\theta :{\rm A}\times {\rm A}\rightarrow {\rm A}$ defined by $
\theta \left( x,y\right) = \{x, y\}$ for all $x,y$ in ${\rm A}$. Then $\theta \in {\rm Z}^{2}\left( 
{\rm A},{\rm A}\right) $ and $\left( 
{\rm A},\cdot ,\{\cdot,\cdot\}_\theta \right) =\left( {\rm A},\cdot ,\{\cdot,\cdot\}
\right)$.

Let $\left( {\rm A},\cdot \right) $ be a derived commutative associative
algebra. The automorphism group, ${\rm Aut}\left( {\rm A}\right) $,
of ${\rm A}$ acts on ${\rm Z}^{2}\left( {\rm A},{\rm A}\right) $ by 
\begin{longtable}{rcl}
$\left(\theta *\phi\right)  \left( x,y\right) $&$=$&$\phi ^{-1}\left( \theta \left( \phi \left(
x\right) ,\phi \left( y\right) \right) \right),$
\end{longtable}
\noindent for $\phi \in {\rm Aut}%
\left( {\rm A}\right) $, and $\theta \in {\rm Z}^{2}\left( {\rm A}, {\rm A}\right) $.

\begin{lemma}
\label{isom}Let $\left( {\rm A},\cdot \right) $ be a derived commutative associative algebra
and $\theta ,\vartheta \in {\rm Z}^{2}\left( {\rm A},{\rm A}\right) $.
Then $\left( {\rm A},\cdot ,\{\cdot,\cdot\}_{\theta }\right) $ and $%
\left( {\rm A},\cdot ,\{\cdot,\cdot\}_{\vartheta }\right) $ are
isomorphic if and only if there is a $\phi \in {\rm Aut}\left( {\rm A}
\right) $ with $\theta *\phi =\vartheta $.
\end{lemma}

\begin{proof}
If $\theta * \phi =\vartheta $, then $\phi :\left( {\rm A},\cdot ,\{\cdot,\cdot\}_{\vartheta }\right) \rightarrow $ $\left( {\rm A},\cdot,\{\cdot,\cdot\}_{\theta }\right) $ is an isomorphism since 
\begin{center}
    $\phi
\left( \vartheta \left( x,y\right) \right) =\theta \left( \phi \left(
x\right) ,\phi \left( y\right) \right) $.
\end{center} On the other hand, if $\phi
:\left( {\rm A},\cdot ,\{\cdot,\cdot\}_{\vartheta }\right)
\rightarrow $ $\left( {\rm A},\cdot ,\{\cdot,\cdot\} _{\theta
}\right) $ is an isomorphism of CPAs, then $\phi \in {\rm Aut}%
\left( {\rm A}\right) $ and $\phi \left(  \{x,y\} _{\vartheta
} \right) = \{\phi \left( x\right) , \phi \left( y\right)\} _{\theta }  
 $. Hence \begin{longtable}{lclcl}
    $\vartheta \left( x,y\right) $&$=$&$\phi ^{-1}\left( \theta
\left( \phi \left( x\right) ,\phi \left( y\right) \right) \right) $&$=$&$\left(\theta *
\phi \right)\left( x,y\right) $
 \end{longtable} \noindent and therefore $\theta * \phi=\vartheta $.
\end{proof}

Consequently, we have a procedure to classify CPAs with given
associated derived commutative associative algebra $\left( {\rm A},\cdot \right) 
$. It consists of three steps:

\begin{enumerate}
\item Compute ${\rm Z}^{2}\left( {\rm A},{\rm A}\right) $.
\item Find the orbits of ${\rm Aut}\left( {\rm A}\right) $ on $%
{\rm Z}^{2}\left( {\rm A},{\rm A}\right) $.
\item  Choose a representative $\theta$ from each orbit and then construct the CPA  $\left( {\rm A},\cdot, \{\cdot,\cdot\}_{\theta }\right) $.
\end{enumerate}

\subsubsection{$3$-dimensional {\rm CPA}s (Theorem A1)}

In the present subsubsection, we give the classification theorem for $3$-dimensional {\rm CPA}s,
which will be proved in the next subsubsection.
We omitted trivial {\rm CPA}s: they are isomorphic to Lie algebras or commutative associative   algebras 
and their algebraic classifications are well-known.

\begin{theoremA1}
Let $({\rm P}, \cdot, \{\cdot,\cdot\})$ be a nontrivial complex $3$-dimensional commutative post-Lie algebra. Then ${\rm P}$ is isomorphic to only one of the
following algebras:

\begin{longtable}{lcl}
 ${\rm P}_{01}$&$:$&$\left\{ 
\begin{tabular}{rrrr}
$e_{1}\cdot e_{1}=e_{1}$ \\ 
$\left\{ e_{2},e_{3}\right\} =e_{2}$%
\end{tabular}%
\right. $\\

 ${\rm P}_{02}^{\alpha }$&$:$&$\left\{ 
\begin{tabular}{rr}
$e_{1}\cdot e_{1}=e_{2}$ &  \\ 
$\left\{ e_{1},e_{2}\right\} =e_{2}+e_{3}$ & $\left\{ e_{1},e_{3}\right\}
=\alpha e_{2}$%
\end{tabular}%
\right. $\\

 ${\rm P}_{03}$&$:$&$\left\{ 
\begin{tabular}{rr}
$e_{1}\cdot e_{1}=e_{2}$ \\ 
$\left\{ e_{1},e_{2}\right\} =e_{3}$%
\end{tabular}%
\right. $\\

 ${\rm P}_{04}$&$:$&$\left\{ 
\begin{tabular}{rr}
$e_{1}\cdot e_{1}=e_{2}$ &  \\ 
$\left\{ e_{1},e_{2}\right\} =e_{3}$ & $\left\{ e_{1},e_{3}\right\} =e_{2}$%
\end{tabular}%
\right. $\\

 ${\rm P}_{05}^{\alpha }$&$:$&$\left\{ 
\begin{tabular}{rr}
$e_{1}\cdot e_{1}=e_{2}$ &  \\ 
$\left\{ e_{1},e_{2}\right\} = e_{2}$ & $\left\{ e_{1},e_{3}\right\}
=\alpha e_{3}$%
\end{tabular}%
\right. $\\

${\rm P}_{06}$&$:$&$\left\{ 
\begin{tabular}{rr}
$e_{1}\cdot e_{1}=e_{2}$ &  \\ 
$\left\{ e_{1},e_{2}\right\} =e_{2}$ & $\left\{ e_{1},e_{3}\right\}
=e_{2}+e_{3}$%
\end{tabular}%
\right. $\\

${\rm P}_{07}$&$:$&$\left\{ 
\begin{tabular}{rr}
$e_{1}\cdot e_{1}=e_{2}$ \\ 
$\left\{ e_{1},e_{3}\right\} =e_{2}$%
\end{tabular}%
\right. $\\

${\rm P}_{08}^{\alpha \neq 0}$&$:$&$\left\{ 
\begin{tabular}{rr}
$e_{1}\cdot e_{2}=e_{3}$ \\ 
$\left\{ e_{1},e_{2}\right\} =\alpha e_{3}$%
\end{tabular}%
\right. $\\

${\rm P}_{09}^{\alpha \neq 0}$&$:$&$\left\{ 
\begin{tabular}{rr}
$e_{1}\cdot e_{1}=e_{2}$ & $e_{1}\cdot e_{2}=e_{3}$ \\ 
$\left\{ e_{1},e_{2}\right\} =\alpha e_{3}$ & 
\end{tabular}%
\right. $\\

${\rm P}_{10}$&$:$&$\left\{ 
\begin{tabular}{rr}
$e_{1}\cdot e_{1}=e_{1}$ & $e_{2}\cdot e_{2}=e_{3}$ \\ 
$\left\{ e_{2},e_{3}\right\} =e_{3}$ & 
\end{tabular}%
\right. $\\

${\rm P}_{11}$&$:$&$\left\{ 
\begin{tabular}{rr}
$e_{1}\cdot e_{1}=e_{1}$ & $e_{2}\cdot e_{3}=e_{2}$ \\ 
$\left\{ e_{2},e_{3}\right\} =-e_{2}$ & 
\end{tabular}%
\right. $\\

${\rm P}_{12}^{\alpha }$&$:$&$\left\{ 
\begin{tabular}{rr}
$e_{1}\cdot e_{3}=e_{1}$ &  \\ 
$\left\{ e_{1},e_{3}\right\} =-e_{1}$ & $\left\{ e_{2},e_{3}\right\}
=\alpha e_{2}$%
\end{tabular}%
\right. $\\

 ${\rm P}_{13}^{\alpha \neq 0}$&$:$&$\left\{ 
\begin{tabular}{rr}
$e_{1}\cdot e_{3}=e_{1}$ & $e_{2}\cdot e_{3}=\alpha e_{2}$ \\ 
$\left\{ e_{1},e_{3}\right\} =-e_{1}$ & $\left\{ e_{2},e_{3}\right\}
=-\alpha e_{2}$%
\end{tabular}%
\right. $\\

${\rm P}_{14}^{\alpha }$&$:$&$\left\{ 
\begin{tabular}{rr}
$e_{1}\cdot e_{3}=e_{1}$ & $e_{3}\cdot e_{3}=e_{2}$ \\ 
$\left\{ e_{1},e_{3}\right\} =-e_{1}$ & $\left\{ e_{2},e_{3}\right\}
=\alpha e_{2}$%
\end{tabular}%
\right. $\\

${\rm P}_{15}$&$:$&$\left\{ 
\begin{tabular}{rr}
$e_{1}\cdot e_{3}=e_{1}+e_{2}$ & $e_{2}\cdot e_{3}=e_{2}$ \\ 
$\left\{ e_{1},e_{3}\right\} =-e_{1}-e_{2}$ & $\left\{ e_{2},e_{3}\right\}
=-e_{2}$%
\end{tabular}%
\right. $\\


 ${\rm P}_{16}$&$:$&$\left\{ 
\begin{tabular}{rr}
$e_{1}\cdot e_{3}=e_{1}$ & $e_{2}\cdot e_{2}=e_{1}$ \\ 
$\left\{ e_{1},e_{3}\right\} =-e_{1}$ & $\left\{ e_{2},e_{3}\right\}
=-e_{2}$%
\end{tabular}%
\right. $
\end{longtable}
\noindent
All listed algebras are non-isomorphic except: ${\rm P}_{08}^{\alpha }\cong 
{\rm P}_{08}^{-\alpha }$.
\end{theoremA1}

\subsubsection{Proof of Theorem A1}

By Proposition \ref{dercommassoc}, we may assume $\left( {\rm P},\cdot \right) $ is one of
the algebras listed in \cite[Theorem A1]{roman2}. 
\begin{enumerate} 
    \item[{\rm (a)}]  If $\left( {\rm P},\cdot \right) $ is
commutative associative, then  ${\rm Z}^{2}\left( {\rm P},{\rm P}\right) \neq \left\{
0\right\} $ if and only if $\left( {\rm P},\cdot \right) \in \big\{ 
{\rm J}_{03},\ {\rm J}_{04},\ {\rm J}_{05},\ {\rm J}_{06}\big\},$ where

\begin{longtable}{lll| llll}
${\rm J}_{03}$ & $:$ & $e_{1}\cdot e_{1} = e_{1}$ & 
${\rm J}_{04}$ & $:$ & $e_{1}\cdot e_{1} = e_{2}$   \\ \hline
${\rm J}_{05}$ & $:$ & $e_{1}\cdot e_{2} = e_{3}$ &  
${\rm J}_{06}$ & $:$ & $e_{1}\cdot e_{1} = e_{2}$ & $e_{1}\cdot e_{2} = e_{3}$ \\ 
\end{longtable}

  \item[{\rm (b)}]  If $\left( {\rm P},\cdot \right) $ is commutative non-associative, then $%
{\rm Z}^{2}\left( {\rm P},{\rm P}\right) \neq \left\{
0\right\}$ if and only if $%
\left( {\rm P},\cdot \right) \in \left\{ {\rm A}_{03}^{0},\ {\rm M}_{04}^{\alpha },\ {\rm M}_{05},\ {\rm M}_{06},\ 
\rm{M}_{07}\right\},$ where

\begin{longtable}{lllllllllll}
${\rm A}_{03}^{0}$ & $:$ & $e_{1}\cdot e_{1} = e_{1}$ & $e_{2}\cdot e_{3} = e_{2}$\\\hline 
 ${\rm M}_{04}^{\alpha}
 $&$:$ & $e_{1}\cdot e_{3}=e_{1}$&$ e_{2}\cdot e_{3}=\alpha e_{2}$ \\\hline
 ${\rm M}_{05}
 $&$:$ & $e_{1}\cdot e_{3}=e_{1}$&$ e_{3}\cdot e_{3}=e_{2}$ \\\hline
 ${\rm M}_{06}
 $&$:$ & $e_{1}\cdot e_{3}=e_{1}+e_{2}$&$ e_{2}\cdot e_{3}=e_{2}$ \\\hline

 ${\rm M}_{07}
 $&$:$&$e_{1}\cdot e_{3}=e_{1}$&$e_{2}\cdot e_{2}=e_{1}$\\
 \end{longtable}

\end{enumerate}
So we study the following cases:

\begin{enumerate}[(A)]
    \item 
\underline{$\left( {\rm P},\cdot \right) ={\rm J}_{03}$}. Let $%
\theta =\left( B_{1},B_{2},B_{3}\right) \neq 0$ be an arbitrary element of $%
{\rm Z}^{2}\left( {\rm J}_{03},{\rm J}_{03}\right) $. Then 
    $\theta  
=\big( 0,\ \alpha _{1}\Delta _{23},\ \alpha _{2}\Delta _{23}\big) $ for some $%
\alpha _{1},\alpha _{2}\in \mathbb{C}$.
 The automorphism group of $\rm{J}_{03}$, $\text{Aut}\left( {\rm J}_{03}\right) $, consists of the
automorphisms $\phi $ given by a matrix of the following form
\[
\phi =%
\begin{pmatrix}
1 & 0 & 0 \\ 
0 & a_{22} & a_{23} \\ 
0 & a_{32} & a_{33}%
\end{pmatrix}%
.
\]%
Let $\phi =\bigl(a_{ij}\bigr)\in $ $\text{Aut}\left( {\rm J}_{03}\right) 
$. Then 
\begin{center}
    $\theta \ast \phi =\big( 0,\ \left( \alpha _{1}a_{33}-\alpha
_{2}a_{23}\right) \Delta _{23},\ \left( \alpha _{2}a_{22}-\alpha
_{1}a_{32}\right) \Delta _{23}\big) $.
\end{center} Since $\theta \neq 0$, we may
assume  $\left( \alpha _{1},\alpha _{2}\right) =\left( 1,0\right) $. So we
get the algebra ${\rm P}_{01}$.

\item \underline{$\left( {\rm P},\cdot \right) ={\rm J}_{04}$}. Let $%
\theta =\left( B_{1},B_{2},B_{3}\right) \neq 0$ be an arbitrary element of $%
{\rm Z}^{2}\left( {\rm J}_{04},{\rm J}_{04}\right) $. Then 
\begin{center}$\theta $ $%
=\left( 0,\ \alpha _{1}\Delta _{12}+\alpha _{2}\Delta _{13},\ \alpha _{3}\Delta
_{12}+\alpha _{4}\Delta _{13}\right) $ for some $\alpha _{1},\alpha
_{2},\alpha _{3},\alpha _{4}\in \mathbb{C}$.\end{center} 
The automorphism group of $%
{\rm J}_{04}$, $\text{Aut}\left( {\rm J}_{04}\right) $, consists of
the automorphisms $\phi $ given by a matrix of the following form
\[
\phi =%
\begin{pmatrix}
a_{11} & 0 & 0 \\ 
a_{21} & a_{11}^{2} & a_{23} \\ 
a_{31} & 0 & a_{33}%
\end{pmatrix}%
.
\]%
Let $\phi =\bigl(a_{ij}\bigr)\in $ $\text{Aut}\left( {\rm J}_{04}\right) 
$. Then 
\begin{center}$\theta \ast \phi =\big( 0,\ \beta _{1}\Delta _{12}+\beta _{2}\Delta_{13},\ \beta _{3}\Delta _{12}+\beta _{4}\Delta _{13}\big),$ where
\end{center}
\begin{longtable}{lcl}
$\beta _{1} $&$=$&${a_{11}}{a^{-1}_{33}}\left( \alpha _{1}a_{33}-\alpha
_{3}a_{23}\right),$ \\
$\beta _{2} $&$=$&${a^{-1}_{11}a^{-1}_{33}}\left( \alpha _{2}a_{33}^{2}-\alpha
_{3}a_{23}^{2}+\alpha _{1}a_{23}a_{33}-\alpha _{4}a_{23}a_{33}\right),$ \\
$\beta _{3} $&$=$&$\alpha _{3}{a_{11}^{3}}{a^{-1}_{33}},$ \\
$\beta _{4} $&$=$&${a_{11}}{a^{-1}_{33}}\left( \alpha _{3}a_{23}+\alpha
_{4}a_{33}\right).$
\end{longtable}

\begin{itemize}
\item $\alpha _{3}\neq 0$ and $\alpha _{1}+\alpha _{4}\neq 0$. We choose $\phi $ to be the following
automorphism:%
\[
\phi =%
\begin{pmatrix}
\frac{1}{\alpha _{1}+\alpha _{4}} & 0 & 0 \\ 
0 & \frac{1}{\left( \alpha _{1}+\alpha _{4}\right) ^{2}} & -\frac{\alpha _{4}%
}{\left( \alpha _{1}+\alpha _{4}\right) ^{3}} \\ 
0 & 0 & \frac{\alpha _{3}}{\left( \alpha _{1}+\alpha _{4}\right) ^{3}}%
\end{pmatrix}%
.
\]%
Then $\theta \ast \phi =\big( 0,\ \Delta _{12}+\alpha \Delta _{13},\ \Delta_{12}\big)$ for some $\alpha \in 
\mathbb{C}
$. So we get the algebras ${\rm P}_{02}^{\alpha }$.

\item $\alpha _{3}\neq 0$ and  $\alpha _{1}+\alpha _{4}=0$. Let $\phi $ be the first of the
following matrices if $\alpha _{1}^{2}+\alpha _{2}\alpha _{3}=0$ or the
second if $\alpha _{1}^{2}+\alpha _{2}\alpha _{3}\neq 0$:%
\[
\begin{pmatrix}
1 & 0 & 0 \\ 
0 & 1 & \alpha _{1} \\ 
0 & 0 & \alpha _{3}%
\end{pmatrix}%
, \
\allowbreak 
\begin{pmatrix}
\left( \alpha _{1}^{2}+\alpha _{2}\alpha _{3}\right) ^{-\frac{1}{2}} & 0 & 0 \\ 
0 & \left( \alpha _{1}^{2}+\alpha _{2}\alpha _{3}\right) ^{-1} & \alpha _{1} \left( \alpha _{1}^{2}+\alpha _{2}\alpha _{3}\right) ^{-\frac{3}{2}} \\ 
0 & 0 & \alpha _{3} \left( \alpha _{1}^{2}+\alpha _{2}\alpha _{3}\right) ^{-\frac{3}{2}}
\end{pmatrix}%
.
\]%
Then $\theta \ast \phi \in \big\{ \big( 0,\ 0,\ \Delta _{12}\big), \ 
\big(0,\ \Delta _{13},\ \Delta _{12}\big) \big\} $. So we get the algebras $%
{\rm P}_{03}$ and ${\rm P}_{04}$.

\item $\alpha _{3}=0$ and $\alpha _{4}\neq 0$.

\begin{itemize}
\item $\alpha _{1}-\alpha _{4}\neq 0$.  Let $\phi $\ be the following
automorphism:%
\[
\phi =%
\begin{pmatrix}
\frac{1}{\alpha _{1}} & 0 & 0 \\ 
0 & \frac{1}{\alpha _{1}^{2}} & \frac{\alpha _{2}}{\alpha _{4}-\alpha _{1}}
\\ 
0 & 0 & 1%
\end{pmatrix}%
.
\]%
Then $\theta \ast \phi =\big( 0,\  \Delta _{12},\ \alpha \Delta _{13}\big)_{\alpha\neq0} $.  So we get the algebras ${\rm P}_{05}^{\alpha\notin \{0,1\} }$.

\item $\alpha _{1}-\alpha _{4}=0$. Let $\phi $\ be the first of the
following matrices if $\alpha _{2}=0$ or the second if $\alpha _{2}\neq 0$:%
\[
\begin{pmatrix}
\frac{1}{\alpha _{1}} & 0 & 0 \\ 
0 & \frac{1}{\alpha _{1}^{2}} & 0 \\ 
0 & 0 & a_{33}%
\end{pmatrix}%
, \ 
\begin{pmatrix}
\frac{1}{\alpha _{1}} & 0 & 0 \\ 
0 & \frac{1}{\alpha _{1}^{2}} & 0 \\ 
0 & 0 & \frac{1}{\alpha _{1}\alpha _{2}}%
\end{pmatrix}%
.
\]%
Then $\theta \ast \phi \in \big\{ \big( 0,\ \Delta _{12},\ \Delta _{13}\big), \ 
\big( 0,\ \Delta _{12}+\Delta _{13},\ \Delta _{13}\big) \big\}$. So we get   ${\rm P}_{05}^1$ and ${\rm P}_{06}$.
\end{itemize}

\item $\alpha _{3}=0$ and $\alpha _{4}=0$. Let $\phi $\ be the first of the following matrices
if $\alpha _{1}=0$ or the second if $\alpha _{1}\neq 0$:%
\[
\begin{pmatrix}
\frac{1}{\alpha _{1}} & 0 & 0 \\ 
0 & \frac{1}{\alpha _{1}^{2}} & -\frac{1}{\alpha _{1}}\alpha _{2} \\ 
0 & 0 & 1%
\end{pmatrix}%
, \ 
\begin{pmatrix}
\alpha _{2} & 0 & 0 \\ 
0 & \alpha _{2}^{2} & 0 \\ 
0 & 0 & 1%
\end{pmatrix}%
.\]%
Then $\theta \ast \phi \in \big\{ \big( 0,\ \Delta _{12},\ 0\big),\ 
\big(0,\ \Delta _{13},\ 0\big) \big\} $. So we get the algebras ${\rm P}_{05}^0$ and ${\rm P}_{07}$.
\end{itemize}

\item \underline{$\left( {\rm P},\cdot \right) ={\rm J}_{05}$}. Let $%
\theta =\left( B_{1},B_{2},B_{3}\right) \neq 0$ be an arbitrary element of $%
{\rm Z}^{2}\left( {\rm J}_{05},{\rm J}_{05}\right) $. Then 
 $\theta  =\big( 0,\ 0,\ \alpha _{1}\Delta _{12}\big) $ for some $\alpha _{1}\in 
\mathbb{C}^{\ast }$. 
The automorphism group of ${\rm J}_{05}$, $\text{Aut}\left( {\rm J}_{05}\right) $, consists of the automorphisms $\phi $
given by a matrix of the following form %
\[
\phi =%
\begin{pmatrix}
\epsilon a_{11} & \varepsilon a_{12} & 0 \\ 
\varepsilon a_{21} & \epsilon a_{22} & 0 \\ 
a_{31} & a_{32} & \epsilon a_{11}a_{22}+\varepsilon a_{12}a_{21}%
\end{pmatrix}%
:\left( \epsilon ,\varepsilon \right) \in \big\{ \left( 1,0\right) ,\left(
0,1\right) \big\} .
\]%
Let $\phi =\bigl(a_{ij}\bigr)\in $ $\text{Aut}\left( {\rm J}_{05}\right) 
$. Then $\theta \ast \phi =\pm \theta $. So we get the algebras ${\rm P}_{08}^{\alpha \neq 0}$.

\item \underline{$\left( {\rm P},\cdot \right) ={\rm J}_{06}$}. Let $\theta =\left( B_{1},B_{2},B_{3}\right) \neq 0$ be an arbitrary element of $%
{\rm Z}^{2}\left( {\rm J}_{06},{\rm J}_{06}\right) $. Then $\theta  
=\left( 0,0,\alpha _{1}\Delta _{12}\right) $ for some $\alpha _{1}\in 
\mathbb{C}^{\ast }$. The automorphism group of ${\rm J}_{06}$, $\text{Aut%
}\left( {\rm J}_{06}\right) $, consists of the automorphisms $\phi $
given by a matrix of the following form
\[
\phi =%
\begin{pmatrix}
a_{11} & 0 & 0 \\ 
a_{21} & a_{11}^{2} & 0 \\ 
a_{31} & 2a_{11}a_{21} & a_{11}^{3}%
\end{pmatrix}%
.
\]%
Let $\phi =\bigl(a_{ij}\bigr)\in $ $\text{Aut}\left( {\rm J}_{06}\right) 
$. Then $\theta \ast \phi =\theta $. So we get the algebras ${\rm P}_{09}^{\alpha \neq 0}$.

\item 
\underline{$\left( {\rm P},\cdot \right) ={\rm J}_{11}$}. Let $%
\theta =\left( B_{1},B_{2},B_{3}\right) \neq 0$ be an arbitrary element of $%
{\rm Z}^{2}\left( {\rm J}_{11},{\rm J}_{11}\right) $. Then 
    $\theta 
=\big( 0,\ 0,\ \alpha _{1}\Delta _{23}\big) $ for some $\alpha _{1}\in 
\mathbb{C}^{\ast }$.
 The automorphism group of ${\rm J}_{11}$, $\text{Aut%
}\left( {\rm J}_{11}\right) $, consists of the automorphisms $\phi $ given by a matrix of the following form 
\[
\phi =%
\begin{pmatrix}
1 & 0 & 0 \\ 
0 & a_{22} & 0 \\ 
0 & a_{32} & a_{22}^{2}%
\end{pmatrix}%
.
\]%
Let $\phi =\bigl(a_{ij}\bigr)\in $ $\text{Aut}\left( {\rm J}_{11}\right) 
$. Then $\theta \ast \phi =\big( 0,\ 0,\ \alpha _{1}a_{22}\Delta _{23}\big)$. 
So we get the algebra ${\rm P}_{10}$.

\item \underline{$\left( {\rm P},\cdot \right) ={\rm A}_{03}^{0}$}. Let $%
\theta =\left( B_{1},B_{2},B_{3}\right) \neq 0$ be an arbitrary element of $%
{\rm Z}^{2}\left( {\rm A}_{03}^{0},{\rm A}_{03}^{0}\right) $. Then 
 $\theta =\big( 0, \ -\Delta _{23},\ 0\big) $. 
Hence we get the algebra ${\rm P}_{11}$.

\item \underline{$\left( {\rm P},\cdot \right) ={\rm M}_{04}^{0}$}. Let $%
\theta =\left( B_{1},B_{2},B_{3}\right) \neq 0$ be an arbitrary element of $%
{\rm Z}^{2}\left( {\rm M}_{04}^{0},{\rm M}_{04}^{0}\right) $. Then 
 $\theta =\big( -\Delta _{13},\ \alpha _{1}\Delta _{23},\ 0\big)$
for some $\alpha
_{1}\in \mathbb{C}$. 
The automorphism group of ${\rm M}_{04}^{0}$, $%
\text{Aut}\left( {\rm M}_{04}^{0}\right) $, consists of the
automorphisms $\phi $ given by a matrix of the following form %
\[
\phi =%
\begin{pmatrix}
a_{11} & 0 & 0 \\ 
0 & a_{22} & a_{23} \\ 
0 & 0 & 1%
\end{pmatrix}%
.
\]%
Let $\phi =\bigl(a_{ij}\bigr)\in $ $\text{Aut}\left( {\rm A}%
_{36}^{0}\right) $. Then $\theta \ast \phi =\theta $. So we get the algebras 
${\rm P}_{12}^{\alpha }$.

\item \underline{$\left( {\rm P},\cdot \right) ={\rm M}_{04}^{1}$}. Let $%
\theta =\left( B_{1},B_{2},B_{3}\right) \neq 0$ be an arbitrary element of $%
{\rm Z}^{2}\left( {\rm M}_{04}^{1},{\rm M}_{04}^{1}\right) $. Then 
\begin{center}$\theta =\big( -\Delta _{13},\ -\Delta _{23},\ 0\big) $.
\end{center} Then we get the algebra $%
{\rm P}_{13}^{1}$.

\item \underline{$\left( {\rm P},\cdot \right) ={\rm M}_{04}^{-1}$}. Let $%
\theta =\left( B_{1},B_{2},B_{3}\right) \neq 0$ be an arbitrary element of $%
{\rm Z}^{2}\left( {\rm M}_{04}^{1},{\rm M}_{04}^{1}\right) $. Then \begin{center}$\theta 
=\big( -\Delta _{13},\ \Delta _{23},\ 0\big) $. \end{center} Then we get the algebra $%
{\rm P}_{13}^{-1}$.

\item \underline{$\left( {\rm P},\cdot \right) ={\rm M}_{04}^{\alpha \notin
\{0,\pm 1 \}}$}. Let $\theta =\left( B_{1},B_{2},B_{3}\right) \neq 0$ be an
arbitrary element of ${\rm Z}^{2}\left( {\rm M}_{04}^{\alpha\notin
\{0,\pm 1 \}},%
{\rm M}_{04}^{\alpha \notin
\{0,\pm 1 \}}\right) $. Then 
\begin{center}$\theta =\big(-\Delta _{13},\ -\alpha \Delta _{23},\ 0\big) $ for some $\alpha _{1}\in 
\mathbb{C}$.\end{center} 
The automorphism group of ${\rm M}_{04}^{\alpha\notin
\{0,\pm 1 \}}$, $\text{Aut}\left( {\rm M}_{04}^{\alpha \notin
\{0,\pm 1 \}}\right) $,
consists of the automorphisms $\phi $ given by a matrix of the following
form %
\[
\phi =%
\begin{pmatrix}
a_{11} & 0 & 0 \\ 
0 & a_{22} & 0 \\ 
0 & 0 & 1%
\end{pmatrix}%
.
\]%
Let $\phi =\bigl(a_{ij}\bigr)\in $ $\text{Aut}\left( {\rm A}%
_{36}^{\alpha \neq 0,\pm 1}\right) $. Then $\theta \ast \phi =\theta $. So
we get the algebras ${\rm P}_{13}^{\alpha\notin
\{0,\pm 1 \}}$.

\item \underline{$\left( {\rm P},\cdot \right) ={\rm M}_{05}$}. Let $%
\theta =\left( B_{1},B_{2},B_{3}\right) \neq 0$ be an arbitrary element of $%
{\rm Z}^{2}\left( {\rm M}_{05},{\rm M}_{05}\right) $. Then 
\begin{center}$\theta =\big( -\Delta _{13},\ \alpha _{1}\Delta _{23},\ 0\big) $ for some $\alpha
_{1}\in \mathbb{C}$.\end{center} The automorphism group of ${\rm M}_{05}$, $\text{Aut%
}\left( {\rm M}_{05}\right) $, consists of the automorphisms $\phi $
given by a matrix of the following form %
\[
\phi =%
\begin{pmatrix}
a_{11} & 0 & 0 \\ 
0 & 1 & a_{23} \\ 
0 & 0 & 1%
\end{pmatrix}%
.
\]%
Let $\phi =\bigl(a_{ij}\bigr)\in $ $\text{Aut}\left( {\rm M}_{05}\right) 
$. Then $\theta \ast \phi =\theta $. So we get the algebras ${\rm P}_{14}^{\alpha }.$

\item \underline{$\left( {\rm P},\cdot \right) ={\rm M}_{06}$}. Let $%
\theta =\left( B_{1},B_{2},B_{3}\right) \neq 0$ be an arbitrary element of $%
{\rm Z}^{2}\left( {\rm M}_{06},{\rm M}_{06}\right) $. Then 
\begin{center}$\theta $ $%
=\big( -\Delta _{13},\ -\Delta _{13}-\Delta _{23},\ 0\big) $.
\end{center} 
 So we get the
algebras ${\rm P}_{15}$.


\item \underline{$\left( {\rm P},\cdot \right) ={\rm M}_{07}$}. Let $%
\theta =\left( B_{1},B_{2},B_{3}\right) \neq 0$ be an arbitrary element of $%
{\rm Z}^{2}\left( {\rm M}_{07},{\rm M}_{07}\right) $. Then \begin{center}$\theta  =
\big( -\Delta _{13},\ -\Delta _{23},\ 0\big) $. 
\end{center} 
So we get the algebra $%
{\rm P}_{16}$.
\end{enumerate}

\subsection {Nilpotent {\rm CPA}s}

\subsubsection{The method of classification of {\rm CPA}s with nontrivial annihilator}\label{metnilp}

\begin{definition}
Let ${\rm P}$ be a {\rm CPA}  and let $\mathbb{V}$ be a vector space. Let $\rm{Z}_{I}^{2}\left( {\rm P},{\mathbb{V}}\right) $ be the space of all symmetric bilinear maps $\theta :{\rm P}\times {\rm P}\rightarrow \mathbb{V}$
and  $\rm{Z}_{II}^{2}\left( {\rm P},{\mathbb{V}}\right) $ be the space of all  skew-symmetric bilinear maps $\vartheta :{\rm P}\times {\rm P}\rightarrow 
\mathbb{V}$, such that 
\begin{equation*}
\vartheta\big (\{ x,y \},z\big)+\vartheta\big (\{y,z \}, x\big)+\vartheta \big(\{ z, x \},y\big)\ =\ 0.
\end{equation*}
\end{definition}

\begin{definition}
Let ${\rm P}$ be a {\rm CPA}  and  $\mathbb{V}$ be a vector space. 
Then the pair $\left( \theta ,\vartheta \right) \in \rm{Z}_{I}^{2}\left( {\rm P},{\mathbb{V}}\right) \times 
\rm{Z}_{II}^{2}\left( {\rm P},{\mathbb{V}}\right) $ is called a \emph{%
cocycle} if we have 

\begin{longtable}{lcl}
$\theta \big(\{x, y\},z)$&$=$&$\theta \big(x,y\cdot z\big)-\theta \big(y,x\cdot z\big),$\\ 
$\theta \big(x,\{y,z\}\big)$&$=$&$\vartheta \big(x\cdot y,z\big)+\vartheta \big(y,x\cdot z\big).$
\end{longtable}
\noindent We denote the set of all
cocycles by $\rm{Z}^{2,2}\left( {\rm P},{\mathbb{V}}\right) $. Note
that $\rm{Z}^{2,2}\left( {\rm P},{\mathbb{V}}\right) $ is a vector
space under pointwise operations. 
\end{definition}

From this definition, we introduce the notion of an annihilator extension of
a given CPA  $({\rm P}, \cdot, \{\cdot,\cdot\})$. 
Consider two bilinear maps 
$\theta, \vartheta :{\rm P}\times {\rm P}\rightarrow 
\mathbb{V}$. Denote by ${\rm P}_{\left( \theta ,\vartheta \right) }$,
the vector space ${\rm P}\oplus \mathbb{V}$. Now, we endow this vector
space with two compositions, a bilinear multiplication 
\begin{equation*}
\cdot _{{\rm P}_{\left( \theta ,\vartheta \right) }}\ :\ 
{\rm P}_{\left(
\theta ,\vartheta \right) }  \times {\rm P}_{\left(
\theta ,\vartheta \right) }\rightarrow {\rm P}_{\left( \theta
,\vartheta \right) }
\end{equation*}%
\begin{equation*}
\left( x+u\right) \cdot _{{\rm P}_{\left( \theta ,\vartheta \right)
}}\left( y+v\right) \ :=\  x\cdot_{{\rm P}}y+\theta (x,y)
\end{equation*}%
and a bilinear multiplication 
\begin{equation*}
\{ -,-\} _{{\rm P}_{\left( \theta ,\vartheta \right) }}\ :\ 
{\rm P}_{\left(
\theta ,\vartheta \right) }  \times {\rm P}_{\left(
\theta ,\vartheta \right) } \rightarrow {\rm P}%
_{\left( \theta ,\vartheta \right) }
\end{equation*}%
\begin{equation*}
\{ x+u,y+v\} _{{\rm P}_{\left( \theta ,\vartheta \right)
}}\ :=\ \{ x,y\} _{{\rm P}}+\vartheta (x,y).
\end{equation*}

In these conditions, one can verify that ${\rm P}_{\left( \theta
,\vartheta \right) }$ is a CPA if and only if $\left(
\theta ,\vartheta \right) \in \rm{Z}^{2,2}\left( {\rm P},{\mathbb{V}%
}\right) $. If $\mathrm{dim}{\mathbb{V}}=s$, the CPA $%
{\rm P}_{\left( \theta ,\vartheta \right) }$ is called an $s$-\emph{%
dimensional annihilator extension} of ${\rm P}$ by $\mathbb{V}$.

\medskip

Let $({\rm P}, \cdot, \{\cdot,\cdot\})$ be a CPA, let $\mathbb{V}$ be a vector
space and let $\left( \theta ,\vartheta \right) \in \rm{Z}^{2,2}\left( 
{\rm P},{\mathbb{V}}\right) $. To study the annihilator of the CPA ${\rm P}_{\left( \theta ,\vartheta \right) }$, we define
the following vector spaces 
\begin{longtable}{rcl}
${\rm{Rad}}(\theta )$&$:=$&$\big\{ x\in {\rm P}:\theta (x,{\rm P})=0\big\},$ \\
${\rm{Rad}}(\vartheta)$&$:=$&$\big\{ x\in {\rm P}:\vartheta (x,{\rm P})=0\big\},$ \\
${\rm{Rad}}( \theta ,\vartheta ) $&$:=$&${\rm{Rad}}(\theta )\cap {\rm{Rad}}(\vartheta ).$
\end{longtable}%
In this context, we have that $\text{\textrm{Ann}}\left( {\rm P}_{\left(
\theta ,\vartheta \right) }\right) =\left( \text{\textrm{Rad}}\left( \theta
,\vartheta \right) \cap \text{\textrm{Ann}}\left( {\rm P}\right) \right)
\oplus \mathbb{V}.$

\medskip

As in \cite{abcf2}, we can prove that if ${\rm P}$ is an $n$-dimensional CPA with $\text{dim}(\text{\textrm{Ann}}({\rm P}))=m\neq 0$, then there exists, up to isomorphism, a unique $(n-m)$%
-dimensional CPA ${\rm P}^{\prime }$ and a cocycle $%
\left( \theta ,\vartheta \right) \in \rm{Z}^{2,2}\left({\rm P}^{\prime}, \text{\textrm{Ann}}\left({\rm P}\right) \right) $ with $%
\mathrm{Rad}\left( \theta ,\vartheta \right) \cap \text{\textrm{Ann}}\left( 
{\rm P}^{\prime }\right) =0$ such that ${\rm P}\cong {\rm P}%
_{\left( \theta ,\vartheta \right) }^{\prime }$ and ${\rm P}/\text{%
\textrm{Ann}}\left({\rm P}\right) \cong {\rm P}^{\prime }$.

\medskip

We denote, as usually, the set of all linear maps from ${\rm P}$ to $%
\mathbb{V}$ by $\text{Hom}\left( {\rm P},{\mathbb{V}}\right) $. For $%
f\in \text{Hom}\left({\rm P},{\mathbb{V}}\right) $, consider the maps 
\begin{equation*}
\delta _{I}f:{\rm P}\times {\rm P}\rightarrow \mathbb{V}\quad
\quad \delta _{II}f:{\rm P}\times {\rm P}
\rightarrow \mathbb{V}
\end{equation*}
\noindent defined by $\delta _{I}f(x,y):=f(x\cdot y)$ and $\delta _{II}f(x,y):=f(\{ x,y\})$, respectively, for any $x,y\in {\rm P}$. Then
it is easy to see that $\delta f:=\left( \delta _{I}f,\delta _{II}f\right)
\in \rm{Z}^{2,2}\left({\rm P},{\mathbb{V}}\right) $. From here, we
can define the vector space 
\begin{longtable}{lcl}
$\rm{B}^{2,2}\left( \rm{P},{\mathbb{V}}\right) $&$:=
$&$\big\{ \delta
f=(\delta _{I}f,\delta _{II}f) \ :\ f\in \mathrm{Hom}\left( \rm{P},{\mathbb{V%
}}\right) \big\}.$
\end{longtable}

It follows that $\rm{B}^{2,2}\left( \rm{P},{\mathbb{V}}\right) $ )
is a linear subspace of $\rm{Z}^{2,2}\left( \rm{P},{\mathbb{V}}%
\right) $. Moreover, we have that if $( \theta ,\vartheta )
,( \theta ^{\prime },\vartheta ^{\prime }) \in \rm{B}^{2,2}\left( \rm{P},{\mathbb{V}}\right) $, then $\rm{P}_{\left(
\theta ,\vartheta \right) }\cong \rm{P}_{\left( \theta ^{\prime},\vartheta ^{\prime }\right) }$. Therefore, it is natural to introduce the
cohomology space $\rm{H}^{2,2}\left( \rm{P},{\mathbb{V}}\right) $
as the quotient space 
\begin{equation*}
\rm{Z}^{2,2}\left( \rm{P},{\mathbb{V}}\right) \big/\rm{B}^{2,2}\left( \rm{P},{\mathbb{V}}\right) ,
\end{equation*}
being clear that the fact $\big[ ( \theta ,\vartheta ) \big] =\big[ ( \theta ^{\prime },\vartheta ^{\prime }) \big] $ in $
\rm{H}^{2,2}\left( \rm{P},{\mathbb{V}}\right) $ implies that $%
\rm{P}_{\left( \theta ,\vartheta \right) }\cong \rm{P}_{\left(
\theta ^{\prime },\vartheta ^{\prime }\right) }$. However, it is possible
that two different elements $\big[ ( \theta ,\vartheta ) \big]
\neq \big[ ( \theta ^{\prime },\vartheta ^{\prime }) \big] $
in $\rm{H}^{2,2}\left(\rm{P},{\mathbb{V}}\right) $ give rise to
two isomorphic CPAs $\rm{P}_{\left( \theta,\vartheta \right) }$ and $\rm{P}_{\left( \theta ^{\prime },\vartheta
^{\prime }\right) }$. We can solve this problem by restricting our study to CPAs without annihilator components and by introducing
"coordinates"  in $\rm{H}^{2,2}\left( 
\rm{P},{\mathbb{V}}\right) $ in the below sense.

\medskip First, we recall that given a CPA $\rm{P}$
if ${\rm P}={\rm I}\oplus {\mathbb C}x$ is a direct sum of two ideals
such that $x\in \text{\textrm{Ann}}\left( \rm{P}\right) $, then $%
\mathbb{C}x$ is called an \emph{annihilator component} of $\rm{P}$.
Algebras without  {annihilator components} we will call split.
Observe that if $\rm{P}$ has an annihilator component $\mathbb{C}x$,
then this reduces the study of $\rm{P}$ to those $\rm{I}$ of
dimension $\mathrm{dim}(\rm{P})-1$. \medskip Second, let $%
v_{1},v_{2},\ldots ,v_{s}$\ be a fixed basis of $\mathbb{V}$ and $\left(
\theta ,\vartheta \right) \in \rm{Z}^{2,2}\left( \rm{P},{\mathbb{V}%
}\right) $. Then $\left( \theta ,\vartheta \right) $ can be uniquely written
as 
\begin{center}
$\big( \theta \left( x,y\right) ,\vartheta \left( a,b\right) \big) \ =\ 
\overset{s}{\sum\limits_{i=1}}\big(\theta _{i}(x,y)
v_{i},\vartheta _{i}(a,b) v_{i}\big),$
\end{center}
where $\left( \theta _{i},\vartheta _{i}\right) \in \rm{Z}^{2,2}\left( 
\rm{P},{\mathbb{C}}\right) $ and \textrm{Rad}$\left( \theta ,\vartheta
\right) =\underset{i=1}{\overset{s}{\cap }}$\textrm{Rad}$\left( \theta
_{i},\vartheta _{i}\right) $. Moreover, $\left( \theta ,\vartheta \right) \in 
\rm{B}^{2,2}\left(\rm{P},{\mathbb{V}}\right) $ if and only if all 
$\left( \theta _{i},\vartheta _{i}\right) \in \rm{B}^{2,2}\left(\rm{P},{\mathbb{C}}\right) $ for all $i=1,2,\ldots ,s$. Further if $
\left\{ e_{1},e_{2},\ldots ,e_{m}\right\} $ is a basis of $\rm{P}^{\left( 1\right) }=\rm{P}\cdot \rm{P}+\{ \rm{P},\rm{P}\} $, then the set $\left\{ \delta e_{1}^{\ast },\delta
e_{2}^{\ast },\ldots ,\delta e_{m}^{\ast }\right\}$, where $e_{i}^{\ast
}(e_{j}):=\delta _{ij}$ and $\delta _{ij}$ is the Kronecker's symbol, is a basis of $\rm{B}^{2,2}\left(\rm{P},{\mathbb{C}}\right)$. As in 
\cite{abcf2}, we can prove that if 
\begin{center}
$\big(\theta(x,y),\vartheta(a,b)\big) \ =\ 
\overset{s}{\sum\limits_{i=1}}\big( \theta _{i}\left( x,y\right)
v_{i},\vartheta_{i}(a,b) v_{i}\big) \in \rm{Z}
^{2,2}(\rm{P},{\mathbb{V}})$ 
\end{center}
and $\mathrm{Rad}\left( \theta ,\vartheta \right) \cap \text{\textrm{Ann}}(\rm{P}) =0$, then $\rm{P}_{( \theta ,\vartheta) }$ has an annihilator component if and only if 
$\big[ ( \theta_{1},\vartheta _{1}) \big],$ $\big[ ( \theta _{2},\vartheta_{2}) \big],$ $\ldots,$ $\big[ ( \theta _{s},\vartheta _{s}) \big]$ are linearly dependent in $\rm{H}^{2,2}(\rm{P},{\mathbb{C}})$.

\medskip Third, denote by $\text{Aut}\left( \rm{P}\right) $ the
automorphism group of the CPA $\rm{P}$. Now, for an
automorphism $\phi \in \text{Aut}\left( \rm{P}\right) $ and a cocycle $%
\left( \theta ,\vartheta \right) \in \rm{Z}^{2,2}( \rm{P},\mathbb{V}) $, define 
\begin{longtable}{lclclclcl}
    $\phi \theta $&$:$&$\rm{P}\times \rm{P} $&$\rightarrow$&$ \mathbb{V}$ & by & $\phi \theta(x,y) $&$:=$&$\theta(\phi(x),\phi(y)),$\\ 
$\phi\vartheta $&$:$&$\rm{P}\times \rm{P}$&$\rightarrow$&$ \mathbb{V}$ &by &$\phi \vartheta(a,b) $&$:=$&$\vartheta(\phi(a),\phi(b))$,
\end{longtable} \noindent for
any $x,y,a,b\in \rm{P}$. Then $\phi(\theta,\vartheta)
=(\phi\theta,\phi \vartheta) \in \rm{Z}^{2,2}(\rm{P}, {\mathbb{V}})$. This is an action of the group $\text{Aut}(\rm{P})$ on $\rm{Z}^{2,2}(\rm{P},{\mathbb{
V}})$. Moreover, $\rm{B}^{2,2}(\rm{P},{\mathbb{V}})$ is invariant under the action of $\text{Aut}(\rm{P})$ and so we have an induced action of $\text{Aut}(\rm{P})$ on $\rm{H}^{2,2}(\rm{P},{\mathbb{V}}) $.
Assume now that $(\theta,\vartheta) $ and $(\theta^{\prime },\vartheta ^{\prime }) $ are two elements of $\rm{Z}^{2,2}(\rm{P},{\mathbb{V}})$ such that
\begin{longtable}{rcl}
$\big( \theta (x,y) ,\vartheta(a,b) \big)$  &$=$&
$\overset{s}{\sum\limits_{i=1}}\left(\theta_{i}(x,y)
v_{i},\vartheta _{i}(a,b) v_{i}\right),$ \\
$\big( \theta^{\prime }(x,y),\vartheta ^{\prime }(
a,b) \big) $&$=$&$\overset{s}{\sum\limits_{i=1}}\left( \theta_{i}^{\prime }\left( x,y\right) v_{i},\vartheta _{i}^{\prime }(a,b) v_{i}\right).$
\end{longtable}
Suppose that $\rm{P}_{(\theta,\vartheta)}$ has no
annihilator components and 
\begin{center}\textrm{Rad}$(\theta, \vartheta)\cap \text{\textrm{Ann}}(\rm{P}) =\text{Rad}\left( \theta
^{\prime },\vartheta ^{\prime }\right) \cap \text{\textrm{Ann}}\left( 
\rm{P}\right) =0$.\end{center}  
Then the algebra $\rm{P}_{\left( \theta,\vartheta \right) }$ is isomorphic to $\rm{P}_{\left( \theta ^{\prime
},\vartheta ^{\prime }\right) }$ if and only if there exists an automorphism 
$\phi \in \text{Aut}\left( \rm{P}\right) $ such that 
$\big[ \phi
( \theta _{i}^{\prime },\vartheta _{i}^{\prime }) \big] $ with $%
i=1,\ldots ,s$ span the same subspace of $\rm{H}^{2,2}\left( \rm{P},{\mathbb{C}}\right) $ as the $\left[ \left( \theta _{i},\vartheta
_{i}\right) \right] $ with $i=1,\ldots,s$. From here if we consider the 
\emph{Grassmannian} $G_{s}\left( \rm{H}^{2,2}(\rm{P},{\mathbb{C}}) \right)$ (the set of all of the linear subspaces of dimension $s$ in $\rm{H}^{2,2}\left( \rm{P},{\mathbb{C}}\right) $)
and define the set 
\begin{center}
$\rm{R}_{s}\left(\rm{P}\right) \ := \ 
\left\{ \mathbb{W}=\big\langle %
\big[(\theta _{1},\vartheta _{1}) \big], \ldots,\big[(\theta_{s},\vartheta_{s}) \big] \big\rangle \in G_{s}(\rm{H}^{2,2}(\rm{P},{\mathbb{C}})):\underset{i=1}{\overset{s}{
\cap }}\mathrm{Rad}(\theta _{i},\vartheta_{i})\cap \mathrm{Ann}(\rm{P})=0\right\},$
\end{center}
then there is a natural action of $\text{Aut}(\rm{P})$ on $\rm{R}_{s}(\rm{P})$ defined as follows: given an
automorphism $\phi \in \text{Aut}\left( \rm{P}\right) $ and a vector
space ${\mathbb{W}}\ =\ \big\langle \big[(\theta_{1},\vartheta
_{1}) \big],\ldots,\big[(\theta_{s},\vartheta_{s})\big] \big\rangle \in \rm{R}_{s}(\rm{P})$, then \begin{center}$\phi \mathbb{W}=\big\langle \phi \big[(\theta _{1},\vartheta
_{1}) \big],\ldots,\phi \big[(\theta _{s},\vartheta_{s}) \big]\big\rangle \in \rm{R}_{s}(\rm{P})$.\end{center} We denote the
orbit of $\mathbb{W}\in \rm{R}_{s}(\rm{P})$ under the
action of $\text{Aut}(\rm{P})$ by $\rm{O}(\mathbb{W})$.

\medskip

Finally, let $\rm{K}(\rm{P},{\mathbb{V}})$ denote a set of all {\rm CPA}s without annihilator components, which
are $s$\emph{-}dimensional annihilator extensions of $\rm{P}$ by ${\mathbb{V}}$ and have $s$-dimensional annihilator. Then
{\small \begin{equation*}
{\rm K}({\rm P},{\mathbb{V}}) =\left\{ {\rm P}_{(\theta, \vartheta) }:
\big( \theta(x,y),\vartheta (a,b)\big) =\overset{s}{\sum\limits_{i=1}}\left(
\theta _{i}\left( x,y\right) v_{i},\vartheta_{i}(a,b) v_{i}\right) \mbox{
and }\big\langle \left[(\theta _{i},\vartheta _{i}) \right]
:i=1,\ldots ,s\big\rangle \in \rm{R}_{s}\left( \rm{P}\right)\right\} .
\end{equation*}}
Given $\rm{P}_{\left( \theta ,\vartheta \right) }\in \rm{K}\left( 
\rm{P},{\mathbb{V}}\right) $, let $\left[ \rm{P}_{\left( \theta,\vartheta \right) }\right] $ denote the isomorphism class of $\rm{P}_{(\theta,\vartheta)}$. Now if $\rm{P}_{\left( \theta
,\vartheta \right) },\rm{P}_{\left( \theta ^{\prime },\vartheta^{\prime }\right) }\in \rm{K}\left( \rm{P},{\mathbb{V}}\right) $
such that
\begin{longtable}{ccl}
$\big( \theta(x,y),\vartheta(a,b) \big)$&$=$&$
\overset{s}{\sum\limits_{i=1}}\big(\theta _{i}(x,y)v_{i},\vartheta _{i}(a,b)v_{i}\big),$\\
$\big(\theta^{\prime }(x,y),\vartheta ^{\prime }(
a,b) \big)$&$ =$&$\overset{s}{\sum\limits_{i=1}}\big( \theta
_{i}^{\prime }(x,y) v_{i},\vartheta _{i}^{\prime }(
a,b) v_{i}\big).$
\end{longtable}
\noindent Then $\left[ \rm{P}_{\left( \theta ,\vartheta \right) }\right] =\left[ 
\rm{P}_{\left( \theta ^{\prime },\vartheta ^{\prime }\right) }\right] $
if and only if 
\begin{center}
${\rm O} \big  \langle \big[ ( \theta_{i},\vartheta _{i}) \big] \ :\ i=1,\ldots ,s\big\rangle \ =\ {\rm O}
\big\langle \big[ \left( \theta _{i}^{\prime },\vartheta _{i}^{\prime
}\right) \big] \ :\ i=1,\ldots ,s\big\rangle$.
\end{center}

\medskip

The above results can now be rephrased by asserting the next result:

\begin{theorem}
There exists a one-to-one correspondence between the set of ${\rm Aut}(\rm{P})$-orbits on $\rm{R}_{s}(\rm{P})$ and the set of isomorphism classes of $\rm{K}(\rm{P},{\mathbb{V}})$.
\end{theorem}

By this theorem, we can construct all CPAs of dimension $n$ with $s$-dimensional annihilator, given those of dimension $n-s$, following the next three steps:

\begin{enumerate}
\item For a CPA $\rm{P}$ of dimension $n-s$,
determine $\rm{H}^{2,2}(\rm{P},\mathbb{C}) $ and $\text{Aut}(\rm{P})$.

\item Determine the set of $\text{Aut}(\rm{P}) $-orbits on $\rm{R}_{s}(\rm{P}) $.

\item For each orbit $\rm{O}$, construct the CPA  corresponding to a representative of $\rm{O}$.
\end{enumerate}

\medskip

Before proceeding to the next section of this paper, we have to introduce the following convenient notation. Denote by $\nabla_{ij}\ :\ \rm{P}\times \rm{P}\rightarrow {\mathbb{C}}$ the symmetric bilinear form
given by
\begin{equation*}
\nabla _{ij}(e_{l},e_{m})\ =\ \left\{ 
\begin{tabular}{ll}
$1$ & if $\left\{ i,j\right\} =\left\{ l,m\right\} $, \\ 
$0$ & otherwise.%
\end{tabular}%
\right.
\end{equation*}

and by $\Delta _{ij} \ :\ \rm{P}\times \rm{P}
\rightarrow {\mathbb{C}}$ the trilinear form defined by 
\begin{equation*}
\Delta _{ij}(e_{l},e_{m})\ =\ \left\{ 
\begin{tabular}{ll}
$1$ & if $(i,j)=(l,m)$, \\ 
$-1$ & if $(i,j)=(m,l)$,
\\ 
$0$ & otherwise.
\end{tabular}%
\ \right.
\end{equation*}%

Finally, we recall the following result from \cite{abcf2} on how an automorphism acts on a cocycle in matrix terms.

\begin{remark}
Given an automorphism $\phi =\big(a_{ij}\big)\in \mathrm{Aut}(\rm{P})$ and two cocycles $(\theta ,\vartheta ),(\theta ^{\prime},\vartheta ^{\prime })\in \rm{Z}^{2,2}(\rm{P},{\mathbb{C}}) $ such that $\theta ^{\prime}=\phi \theta $ and $\vartheta^{\prime}=$ $\phi \vartheta $. If we denote by 
\begin{center}$A$, $A^{\prime }$, $B=%
\begin{pmatrix}
B_{1} & B_{2} & \cdots & B_{n}%
\end{pmatrix}%
$ and $B^{\prime }=%
\begin{pmatrix}
B_{1}^{\prime } & B_{2}^{\prime } & \cdots & B_{n}^{\prime }%
\end{pmatrix}%
$\end{center}  the matrix forms of $\theta ,\theta ^{\prime },\vartheta ,\vartheta
^{\prime }$, respectively, then $A^{\prime }=\phi ^{t}A\phi $ and $%
B_{k}^{\prime }=\phi ^{t}\left( \underset{i=1}{\overset{n}{\sum }}%
a_{ik}B_{i}\right) \phi $.
\end{remark}

The following observations are obtained as a direct consequence of Theorem A1.

\begin{proposition}
There are no nontrivial $2$-dimensional nilpotent CPAs.    
All trivial $2$-dimensional nilpotent CPAs are listed below:
\begin{center}
 $\mathbf{P}_{1} : = {\mathbb C}^2;$  \ 
 $\mathbf{P}_{2} : e_1 \cdot e_1 = e_2$.
\end{center}
\end{proposition}



\begin{proposition}
\label{3classif} 
Let ${\rm P}$ be a nontrivial $3$-dimensional nilpotent {\rm CPA}. 
Then ${\rm P}$ is isomorphic to only one of the following algebras:

$\mathbb{P}_{1}^{\alpha\neq 0}:\left\{
\begin{tabular}{l}
$e_1\cdot e_2 = \alpha e_3$ \\
$\left\{ e_1, e_2 \right\} =e_{3}$
\end{tabular}
\right.;$ \ 
 $\mathbb{P}_{2}:\left\{
\begin{tabular}{l}
$e_2 \cdot e_2 = e_3$ \\
$\left\{ e_1, e_2 \right\} =e_{3}$
\end{tabular}
\right.;$ \
 $\mathbb{P}_{3}^{\alpha}: \left\{
\begin{tabular}{ll}
$e_1 \cdot e_1 = e_2$ &$e_1 \cdot e_2 =\alpha e_3$ \\
$\left\{ e_1, e_2 \right\} =e_{3}$
\end{tabular}
\right. $

\noindent
Between these algebras, there are precisely the following isomorphisms: $%
\mathbb{P}_{1}^{\alpha } \cong \mathbb{P}_{1}^{-\alpha }.$
All trivial $3$-dimensional nilpotent CPAs are listed below:

$\mathbb{P}_{1}^{0}: \left\{ e_1, e_2 \right\} =e_{3};$ \ 
 $\mathbb{P}_{4}= {\mathbb C}^3;$ \ 
$\mathbb{P}_{5}=\mathbf{P}_{2}\oplus {\mathbb C};$ \
 $\mathbb{P}_{6}: e_1 \cdot e_2 = e_3;$ \
$\mathbb{P}_{7}: e_1 \cdot e_1 = e_2, \ e_1 \cdot e_2 = e_3$.

\end{proposition}

\subsubsection{$4$-dimensional nilpotent {\rm CPA}s (Theorem A2)}

%
%
%
%
%
%

\begin{theoremA2}
\label{4classif} Let $({\rm P}, \cdot, \{\cdot,\cdot\})$ be a nontrivial non-split $4$-dimensional nilpotent {\rm CPA}. Then ${\rm P}$ is isomorphic to only one of
the following algebras:

\begin{longtable}{lcl}

 ${\mathfrak P}_{01}^{\alpha}$&$:$&$\left\{\begin{tabular}{rrrrrr}
$e_1\cdot e_1 = e_4$&$e_1 \cdot e_2 = \alpha e_3$ \\ 
$\left\{ e_1, e_2 \right\} =e_{3}$
\end{tabular}
\right. $\\

 ${\mathfrak P}_{02}$&$:$&$\left\{\begin{tabular}{rrrrrr}
$ e_1\cdot e_1 = e_4$&$ e_2\cdot e_2 = e_3$ \\ 
$\left\{ e_1, e_2 \right\} =e_{3}$
\end{tabular}
\right. $\\

 ${\mathfrak P}_{03}$&$:$&$\left\{\begin{tabular}{rrrrrr}
$ e_1\cdot e_2= e_4$ \\ 
$\left\{ e_1, e_2 \right\} =e_{3}$
\end{tabular}
\right. $\\

 ${\mathfrak P}_{04}^{\alpha}$&$:$&$\left\{\begin{tabular}{rrrrrrr}
$ e_1\cdot e_1 = e_3$&$ e_1 \cdot e_2 = e_4$&$ e_2\cdot e_2 = \alpha e_3$ \\ 
$\left\{ e_1, e_2 \right\} =e_{3}$
\end{tabular}
\right. $\\

${\mathfrak P}_{05}$&$:$&$\left\{\begin{tabular}{rr}
$ e_1\cdot e_1 = e_2$&$ e_1 \cdot e_2 =e_3$ \\ 
$\left\{ e_1, e_2 \right\} =e_{4}$
\end{tabular}
\right. $ \\

 ${\mathfrak P}_{06}$&$:$&$\left\{
\begin{tabular}{rr}
$e_1\cdot e_1 = e_4$&$ e_2\cdot e_2 = e_4  $ \\
$\left\{ e_1, e_2 \right\} =e_{3}$&$ \left\{ e_1, e_3 \right\} =e_{4} $
\end{tabular}
\right.$\\

 ${\mathfrak P}_{07}$&$:$&$\left\{
\begin{tabular}{rrr}
$e_2\cdot e_2 = e_4  $ \\
$\left\{ e_1, e_2 \right\} =e_{3}$&$ \left\{ e_1, e_3 \right\} =e_{4} $
\end{tabular}
\right.$\\

 ${\mathfrak P}_{08}$&$:$&$\left\{
\begin{tabular}{rrrrrr}
$e_1\cdot e_1 = e_4 $ \\
$\left\{ e_1, e_2 \right\} =e_{3}$&$ \left\{ e_1, e_3 \right\} =e_{4}$
\end{tabular}
\right.$\\

 ${\mathfrak P}_{09}$&$:$&$\left\{
\begin{tabular}{rrrr}
$e_1 \cdot e_2 = e_4$ \\
$\left\{ e_1, e_2 \right\} =e_{3}$&$ \left\{ e_1, e_3 \right\} =e_{4}$
\end{tabular}
\right.$\\

${\mathfrak P}_{10}$&$:$&$\left\{
\begin{tabular}{rrrr}
$e_2 \cdot e_2 = e_3$ \\
$\left\{ e_1, e_2 \right\} =e_{3}$&$\left\{ e_2, e_3 \right\} = e_4$
\end{tabular}
\right.$\\

 ${\mathfrak P}_{11}$&$:$&$\left\{
\begin{tabular}{rrr}
$e_1\cdot e_1 = e_4$&$ e_2 \cdot e_2 = e_3$ \\
$\left\{ e_1, e_2 \right\} =e_{3}$&$\left\{ e_2, e_3 \right\} = e_4$
\end{tabular}
\right.$\\

 ${\mathfrak P}_{12}$&$:$&$\left\{
\begin{tabular}{rrr}
$e_1\cdot e_2 = e_4$&$ e_2 \cdot e_2 = e_3$ \\
$\left\{ e_1, e_2 \right\} =e_{3}$&$\left\{ e_2, e_3 \right\} = e_4$
\end{tabular}
\right.$\\

 ${\mathfrak P}_{13}$&$:$&$\left\{
\begin{tabular}{rrr}
$e_2 \cdot e_2 = e_3$&$ e_2 \cdot e_3 = e_4$  \\
$\left\{ e_1, e_2 \right\} =e_{3}$&$\left\{ e_1, e_3 \right\} = 2e_4$
\end{tabular}
\right.$\\

 ${\mathfrak P}_{14}$&$:$&$\left\{
\begin{tabular}{rrrrrr}
$e_2 \cdot e_2 = e_3+e_4$&$ e_2 \cdot e_3 = e_4$   \\
$\left\{ e_1, e_2 \right\} =e_{3}$&$\left\{ e_1, e_3 \right\} = 2e_4 $
\end{tabular}
\right.$\\

 ${\mathfrak P}_{15}^{\alpha}$&$:$&$\left\{
\begin{tabular}{rrrrr}
$e_1\cdot e_1 = e_4$&$ e_2 \cdot e_2 = e_3+\alpha e_4$&$ e_2 \cdot e_3 = e_4$  \\
$\left\{ e_1, e_2 \right\} =e_{3}$&$\left\{ e_1, e_3 \right\} = 2e_4$
\end{tabular}
\right.$\\

 ${\mathfrak P}_{16}^{\alpha}$&$:$&$\left\{\begin{tabular}{rrrr}
$e_1 \cdot e_1 = e_2$&$e_1\cdot e_2 = \alpha e_3$&$ e_1 \cdot e_3 = \alpha e_4$&$e_2 \cdot e_2 = \alpha\left(\alpha-1\right) e_4$ \\ 
$\left\{ e_1, e_2 \right\} =e_{3}$&$ \left\{ e_1, e_3 \right\} =e_{4}$
\end{tabular}
\right. $\\

 ${\mathfrak P}_{17}$&$:$&$\left\{\begin{tabular}{rrrr}
$e_1 \cdot e_1 = e_2$&$e_1\cdot e_2 = e_4$\\ 
$\left\{ e_1, e_2 \right\} =e_{3}$&$ \left\{ e_1, e_3 \right\} =e_{4}$
\end{tabular}
\right. $\\

 ${\mathfrak P}_{18}$&$:$&$\left\{\begin{tabular}{rrrr}
$e_1 \cdot e_1 = e_2$&$e_1\cdot e_2 = e_3+e_4$&$ e_1 \cdot e_3 = e_4$ \\
$\left\{ e_1, e_2 \right\} =e_{3}$&$ \left\{ e_1, e_3 \right\} =e_{4}$
\end{tabular}
\right. $\\
${\mathfrak P}_{19}$&$:$&$\left\{
\begin{tabular}{rrrr}
$e_1\cdot e_1 = e_4$ \\
$\left\{ e_2, e_3 \right\} =e_{4}$
\end{tabular}
\right.$\\

 ${\mathfrak P}_{20}$&$:$&$\left\{
\begin{tabular}{rr}
$e_1\cdot e_1 = e_4$&$ e_2\cdot e_2 = e_{4}$ \\
$\left\{ e_1, e_3 \right\} =e_{4}$
\end{tabular}
\right.$\\

 ${\mathfrak P}_{21}$&$:$&$\left\{
\begin{tabular}{rrr}
$e_1\cdot e_1 = e_4$&$ e_2\cdot e_2 = -e_{4}$ \\
$\left\{ e_1, e_3 \right\} =e_{4}$&$ \left\{ e_2, e_3 \right\} =e_{4}$
\end{tabular}
\right.$\\

 ${\mathfrak P}_{22}^{\alpha}$&$:$&$\left\{
\begin{tabular}{rrr}
$e_1\cdot e_2 = e_4$&$ e_3\cdot e_3 = e_4$ \\
$\left\{ e_1, e_3 \right\} =e_{4}$&$\left\{ e_2, e_3 \right\}=\alpha e_{4}$
\end{tabular}
\right.$\\

 ${\mathfrak P}_{23}$&$:$&$\left\{
\begin{tabular}{rrr}
$e_1\cdot e_1 = e_2$&$ e_1 \cdot e_2 = e_4$&$ e_3\cdot e_3 = e_4$ \\
$\left\{ e_1, e_3 \right\} =e_{4}$
\end{tabular}
\right.$\\

 ${\mathfrak P}_{24}$&$:$&$\left\{
\begin{tabular}{rr}
$e_1\cdot e_1 = e_2$&$ e_1 \cdot e_2 = e_4$ \\
$\left\{ e_1, e_3 \right\} =e_{4}$
\end{tabular}
\right.$\\

 ${\mathfrak P}_{25}^{\alpha}$&$:$&$\left\{
\begin{tabular}{rrrrr}
$e_1\cdot e_1 = e_2$&$ e_1 \cdot e_2 = \alpha e_4$&$ e_3\cdot e_3 = e_4 $ \\
$\left\{ e_1, e_2\right\} =e_{4} $
\end{tabular}
\right.$\\

 ${\mathfrak P}_{26}^{\alpha}$&$:$&$\left\{
\begin{tabular}{rrrrr}
$e_1\cdot e_1 = e_2$&$ e_1\cdot e_2 = \alpha e_4$&$ e_1\cdot e_3 = e_4$ \\
$\left\{ e_1, e_2 \right\} =e_{4}$
\end{tabular}
\right.$\\

 ${\mathfrak P}_{27}$&$:$&$\left\{ 
\begin{tabular}{rrrrr}
$e_1 \cdot e_1 = e_2$&$ e_1 \cdot e_2 = e_3$&$ e_1 \cdot e_3 = e_4$&$ e_2 \cdot e_2 = e_4$ \\ 
$\left\{ e_1, e_2 \right\} =e_{4}$
\end{tabular}
\right.$\\

\end{longtable}
 
\end{theoremA2}

\subsubsection{Proof of Theorem A2}

The proof of our Theorem is based on the method described in subsubsection \ref{metnilp}.
Namely, we will construct $2$-dimensional annihilator extensions for $2$-dimensional nilpotent {\rm CPA}s
and   $1$-dimensional annihilator extensions for $3$-dimensional nilpotent {\rm CPA}s.
All calculations are given below.

\medskip 

\begin{center}
{\it $\bullet$ $2$-dimensional annihilator extensions of $2$-dimensional nilpotent {\rm CPA}s.}
\end{center}

\begin{proposition}\label{14}
   In the following table, we have computed the cohomology of the $2$-dimensional nilpotent {\rm CPA}s, 
   which will be needed in the calculation of the annihilator extensions.

\begin{longtable}{rcl|rcl}
${\rm Z}^{2,2}\left( \mathbf{P}_{1},{\mathbb{C}}\right) $&$=$&$ 
\big\langle (\nabla_{11}, 0), (\nabla_{12}, 0), (\nabla_{22}, 0), (0,
\Delta_{12}) \big\rangle$ & 
${\rm B}^{2,2}\left( \mathbf{P}_{1},{\mathbb{C}}\right) $&$=$&$ 
$ Trivial \\

${\rm H}^{2,2}\left( \mathbf{P}_{1},{\mathbb{C}}\right) $&$=$&$ 
\big\langle \left[(\nabla_{11}, 0)\right], \left[(\nabla_{12}, 0)\right], \left[%
(\nabla_{22}, 0)\right], \left[(0, \Delta_{12})\right]\big\rangle$ & 
${\rm Ann}( \mathbf{P}_{1}) $&$=$&$ \mathbf{P}_{1}$\\

\hline

${\rm Z}^{2,2}\left( \mathbf{P}_{2},{\mathbb{C}}\right) $&$=$&$ 
\big\langle (\nabla_{11}, 0), (\nabla_{12}, 0), (0,\Delta_{12})\big\rangle$ &

${\rm B}^{2,2}\left( \mathbf{P}_{2},{\mathbb{C}}\right) $&$=$&$ 
\langle (\nabla_{11}, 0) \rangle$\\

${\rm H}^{2,2}\left( \mathbf{P}_{2},{\mathbb{C}}\right) $&$=$&$ 
\big\langle \left[(\nabla_{12}, 0)\right], \left[(0,\Delta_{12})\right]\big\rangle$ & 

${\rm Ann}( \mathbf{P}_{2}) $&$=$&$ \langle e_2\rangle$
\end{longtable}

\end{proposition}

\noindent
{\it Continuation of proof of} {\rm Theorem A2.}
Proposition \ref{14} tells us that we have to consider 
$2$-dimensional annihilator extension of  $\mathbf{P}_{1}$.

\begin{enumerate}

\item 
[$\mathbf{P}_{1}:$] Fix the following notation 
\begin{center}
$\Upsilon_1=\big[(\nabla_{11},\  0)\big], \ 
\Upsilon_2=\big[(\nabla_{12},\  0)\big],\ 
\Upsilon_3=\big[(\nabla_{22},\  0)\big], \ 
\Upsilon_4=\big[(0,\  \Delta_{12})\big].$ \\
\end{center}

Given $\Gamma _{1}=\sum_{i=1}^{4}\alpha _{i}\Upsilon _{i}$ and $\Gamma
_{2}=\sum_{i=1}^{4}\beta _{i}\Upsilon _{i}$ in ${\rm H}^{2,2}\left( \mathbf{P}_{1},\mathbb{C}\right) $ such that $\text{rad}(\Gamma _{1})\cap \text{rad}%
(\Gamma _{2})=0$, consider the action of an automorphism $\phi $, then $\phi
\Gamma _{1}=\sum_{i=1}^{4}\alpha _{i}^{\ast }\Upsilon _{i}$ and $\phi \Gamma
_{2}=\sum_{i=1}^{4}\beta _{i}^{\ast }\Upsilon _{i}$ where
\begin{longtable}{lcl}
$\alpha^* _1 $&$=$&$ \alpha_{1} x_{11}^2+2 \alpha_{2} x_{11} x_{21}+\alpha_{3}
x_{21}^2,$ \\ 
$\alpha^* _2$&$=$&$ \alpha_{1} x_{11} x_{12}+\alpha_{2} x_{11} x_{22}+\alpha_{2}
x_{12} x_{21}+\alpha_{3} x_{21} x_{22},$\\ 
$\alpha^* _3 $&$=$&$ \alpha_{1} x_{12}^2+2 \alpha_{2} x_{12} x_{22}+\alpha_{3}
x_{22}^2,$  \\ 
$\alpha^* _4$&$=$&$ \alpha_{4} x_{11} x_{22}-\alpha_{4} x_{12} x_{21},$  \\ 
$\beta^* _1$&$=$&$ \beta_{1} x_{11}^2+2 \beta_{2} x_{11} x_{21}+\beta_{3} x_{21}^2,$ \\ 
$\beta^* _2 $&$=$&$ \beta_{1} x_{11}x_{12}+\beta_{2} x_{11} x_{22}+\beta_{2} x_{12} x_{21}+\beta_{3} x_{21}x_{22},$ \\
$\beta^* _3 $&$=$&$ \beta_{1} x_{12}^2+2 \beta_{2} x_{12}x_{22}+\beta_{3} x_{22}^2,$\\
$\beta^*_4 $&$=$&$ \beta_{4} x_{11} x_{22}-\beta_{4} x_{12} x_{21}.$
\end{longtable}

We may assume $\beta _{4}=0$ and $\alpha_4\neq0.$ 
Furthermore, 
we may assume 
\begin{center}
    $\left( \beta _{1},\beta _{2},\beta _{3},\beta _{4}\right) \in
\big\{ \left( 1,0,0,0\right) ,\ \left( 0,1,0,0\right) \big\} $.
\end{center} From here, the following cases arise:

\begin{enumerate}
\item $\left( \beta _{1},\beta _{2},\beta _{3},\beta _{4}\right) =\left(
1,0,0,0\right) $. Then we may assume $\alpha _{1}=0$, then we get the representatives 
\begin{longtable}{lcl}
$W_{1}^{\alpha } $&$=$&$ \big\langle \left( \alpha \Upsilon _{2}+\Upsilon _{4},\ \Upsilon _{1}\right):\alpha \in 
\mathbb{C} \big \rangle \mbox{ if }\alpha _{3}=0,$\\ 
$W_{2}$&$=$&$\big\langle \left( \Upsilon _{3}+\Upsilon _{4},\ \Upsilon
_{1}\right) \big\rangle\mbox{ if }\alpha _{3}\neq 0$.
\end{longtable}Furthermore, $W_{1}^{\alpha}$ and $W_{1}^{\beta }$ are in the same orbit if and only if $\alpha =\beta $.

\item $\left( \beta _{1},\beta _{2},\beta _{3},\beta _{4}\right) =\left(
0,1,0,0\right) $. Then we may assume $\alpha _{2}=0$, then  
 
  $\alpha _{1}=\alpha _{3}=0$, we obtain the
representative 
\begin{longtable}{lcl}
$W_{3}$&$=$&$\big\langle \left( \Upsilon _{4},\ \Upsilon _{2}\right)\big \rangle $ if $\left( \alpha _{1},\alpha _{3}\right) = 0,$\\
$W_{4}^{\alpha }$&$=$&$ \big\langle \left( \Upsilon _{1}+\alpha \Upsilon _{3}+\Upsilon
_{4},\ \Upsilon _{2}\right) :\alpha \in \mathbb{C} \big\rangle $ if $\left( \alpha _{1},\alpha _{3}\right) \neq 0.$ 
\end{longtable}
Moreover, $W_{4}^{\alpha }$ and $W_{4}^{\beta }$ are in the same orbit if and only if $\alpha=\beta$.

\end{enumerate}
We denote the corresponding algebras of the above representatives by 
${\mathfrak P}_{01}^{\alpha }$, ${\mathfrak P}_{02}$,  ${\mathfrak P}_{03}$ and ${\mathfrak P}_{04}^{\alpha }$, respectively.

\item[$\mathbf{P}_{2}:$]
Here we get the algebra ${\mathfrak P}_{05}$ corresponding to 
$\big\langle \big[
(\nabla _{12},0)\big], \ \big[ (0,\Delta _{12})\big] \big\rangle $.

\end{enumerate}

\begin{center}
{\it $\bullet$ $1$-dimensional annihilator extensions of $3$-dimensional algebras.}
\end{center}

\begin{proposition}\label{16}
In the following table, we have computed the cohomology of the $3$-dimensional nilpotent commutative post-Lie algebras, which will be needed in the calculation of the annihilator extensions.

\begin{longtable}{|l|c|c|c|}
\hline
$\mathbb{P}$ & ${\rm Ann}, \ {\rm B}^{2,2}$ &  

${\rm Z}^{2,2}, \ {\rm H}^{2,2}$      \\ \hline

$\mathbb{P}_{1}^{0}$ & 
$\langle e_{3}\rangle $ &

$
\big\langle (\nabla _{11},0),(\nabla _{12},0),(\nabla _{22},0), (0,\Delta _{12}),(0,\Delta _{13}),(0,\Delta _{23})\big \rangle $\\
 
&$\big\langle (0,\Delta _{12})\big\rangle $&   $%
\big\langle [ (\nabla _{11},0)],\ 
[ (\nabla _{12},0)],
[ (\nabla _{22},0)],  
[ (0,\Delta _{13})] ,
[ (0,\Delta _{23})] \big\rangle%
$   \\ \hline

$\mathbb{P}_{1}^{\alpha \neq 0}$ & 
$\langle e_{3}\rangle $ &  

$\big\langle (\nabla _{11},0),(\nabla _{12},0),(\nabla _{22},0), (0,\Delta _{12})\big\rangle$ \\
&
$\big\langle (\alpha \nabla _{12},\Delta _{12})\big\rangle $ &  $
\big\langle [ (\nabla _{11},0)],
[ (\nabla _{12},0)],
[ (\nabla _{22},0)] 
\big\rangle$ 
 \\ \hline

$\mathbb{P}_{2}$ & 
$\langle e_{3}\rangle $ &

$\big\langle (\nabla _{11},0),(\nabla _{12},0),(\nabla _{22},0), (\nabla _{23},2\Delta _{13}),(0,\Delta _{12}),(0,\Delta _{23})\big\rangle$\\
&$\big\langle (\nabla _{22},\Delta _{12}) \big\rangle $& 
$\big\langle [ (\nabla _{11},0)],
[ (\nabla _{12},0)],
[ (\nabla _{22},0)],  
[ (\nabla _{23},2\Delta _{13})],
[ (0,\Delta _{23})]
\big\rangle$
\\ \hline

$\mathbb{P}_{3}^{\alpha }$ & 
$\langle e_{3}\rangle $ &

$\big\langle 
(\nabla _{11},0),(\nabla _{12},0), 
(\alpha \left( \alpha -1\right) \nabla _{13}+\alpha \nabla _{22},
\Delta_{13}), (0,\Delta _{12})\big\rangle$\\
&$\left\langle 
(\alpha \nabla _{12},\Delta _{12})  \right\rangle $& 
$\big\langle [ (\nabla _{12},0)],
[ (\alpha \left( \alpha-1\right) \nabla _{13}+\alpha \nabla _{22},\Delta _{13})] \big\rangle$\\ \hline

$\mathbb{P}_{4}$ & 
$\mathbb{P}_{4}$ & 
 
$\left\langle 
\begin{array}{c}
(\nabla _{11},0),(\nabla _{12},0),(\nabla _{13},0),  (\nabla _{22},0),(\nabla _{23},0),\\ 
(\nabla _{33},0), 
(0,\Delta _{12}),(0,\Delta _{13}),(0,\Delta _{23}) 
\end{array} \right\rangle$ \\
& {\rm Trivial}& 
$\left\langle 
\begin{array}{c} 
[ (\nabla _{11},0)] ,[ (\nabla _{12},0)], [ (\nabla _{13},0)],
[(\nabla _{22},0)], \\ {} 
[(\nabla _{23},0)] , 
[(\nabla _{33},0)] , [ (0,\Delta _{12})] ,[ (0,\Delta _{13})] ,[(0,\Delta _{23})] \end{array} \right\rangle$ \\ \hline

$\mathbb{P}_{5}$ &
$\langle e_{2},e_{3}\rangle $ &

$\big\langle (\nabla _{11},0),(\nabla _{12},0),(\nabla _{13},0), 
(\nabla _{33},0),(0,\Delta _{12}),(0,\Delta _{13})\big\rangle$ \\

 &$\big\langle (\nabla _{11},0) \big\rangle $&  $
\big\langle [ (\nabla _{12},0)] ,[ (\nabla _{13},0)] ,%
[ (\nabla _{33},0)] ,  
[ (0,\Delta _{12})] ,[ (0,\Delta _{13})] \big\rangle$
\\ \hline

$\mathbb{P}_{6}$ & $\langle e_{3}\rangle $ &

$\big\langle (\nabla _{11},0),(\nabla _{12},0),(\nabla _{22},0), 
(0,\Delta _{12})\big\rangle$\\
& $\big\langle (\nabla _{12},0)\big\rangle $ & 
$\big\langle [ (\nabla _{11},0)] ,[ (\nabla _{22},0)] ,%
[ (0,\Delta _{12})] \big\rangle$
 \\ \hline

$\mathbb{P}_{7}$ & 
$\langle e_{3}\rangle $ &

$\big\langle (\nabla _{11},0),(\nabla _{12},0), 
(\nabla _{13}+\nabla _{22},0),(0,\Delta _{12})\big\rangle$\\
&$\big\langle 
 (\nabla _{11},0), 
(\nabla _{12},0)
 \big\rangle $ & 
$\big\langle [ (\nabla _{13}+\nabla _{22},0)] ,[ (0,\Delta _{12})%
] \big\rangle$
 \\ \hline

\end{longtable}

\end{proposition}

\noindent 
{\it Continuation of proof of} {\rm Theorem A2.}
Proposition \ref{16} tells us that we have to consider 
$1$-dimensional annihilator extension of 
$\mathbb{P}_{1}^{0}$, $\mathbb{P}_{2}$, $\mathbb{P}_{3}^{\alpha}$, $\mathbb{P}_{4}$, $\mathbb{P}_{5}$, and  $\mathbb{P}_{7}$.

\begin{enumerate}

\item[$\mathbb{P}_{1}^{0}:$]
Fix the following notations 
\begin{center} 
$\Upsilon_1=\big[(\nabla_{11},\  0)\big], \ 
\Upsilon_2=\big[(\nabla_{12},\  0)\big],  \ 
\Upsilon_3=\big[(\nabla_{22},\  0)\big], \ 
\Upsilon_4=\big[(0,\  \Delta_{13})\big], \ 
\Upsilon_5=\big[(0,\  \Delta_{23})\big].$  
\end{center}

Given $\Gamma = \sum_{i=1}^5\alpha_i \Upsilon_i$ in ${\rm H}^{2,2}\left( 
\mathbb{P}_{1}^{0},\mathbb{C}\right)$ such that $\text{rad}%
(\Gamma)\cap\text{Ann}(\mathbb{P}_{1}^{0})=0$. Since $\text{Ann}(\mathbb{P}_{1}^{0})=\langle e_3\rangle$, we may assume $(\alpha_4,\alpha_5)\neq(0, 0)$.
On the other hand, for having a nontrivial algebra, we have to assume $(\alpha_1, \alpha_2,\alpha_3)\neq(0,0,0)$.
The automorphism group of $\mathbb{P}_{1}^{0}$, $\text{Aut}\left( \mathbb{P}_{1}^{0}\right) $, consists of
the automorphisms $\phi $ given by a matrix of the following form %
\[
\phi =%
\left(\begin{array}{ccc}
 x_{11} &  x_{12} & 0 \\
 x_{21} &  x_{22} & 0 \\
 x_{31} &  x_{32} & x_{11}x_{22}-x_{12}x_{21} \\
\end{array} \right).
\]%

Now, consider the action of an automorphism $\phi,$ $\phi \Gamma = \sum_{i=1}^5\alpha^*_i \Upsilon_i$ where

\begin{longtable}{lcl}
$\alpha^* _1$&$=$&$ \alpha_{1} x_{11}^2+2 \alpha_{2} x_{11} x_{21}+\alpha_{3}x_{21}^2,$ \\ 
$\alpha^* _2$&$=$&$ \alpha_{1} x_{11} x_{12}+\alpha_{2}( x_{11} x_{22}+x_{12} x_{21})+\alpha_{3} x_{21} x_{22},$ \\ 
$\alpha^* _3$&$=$&$ \alpha_{1} x_{12}^2+2 \alpha_{2} x_{12} x_{22}+\alpha_{3}x_{22}^2,$ \\ 
$\alpha^* _4$&$=$&$ (x_{11} x_{22}-x_{12} x_{21}) (\alpha_{4} x_{11}+\alpha_{5}x_{21}),$ \\ 
$\alpha^* _5$&$=$&$ (x_{11} x_{22}-x_{12} x_{21}) (\alpha_{4} x_{12}+\alpha_{5}x_{22}).$ \\ 
\end{longtable}
 
We  may assume  that $\left(\alpha_{4},\alpha_{5}\right)=\left(1,0%
\right)$. Then, we have  the following cases:

\begin{enumerate}
\item $\alpha _{3}\neq 0$. Then we have the representative $\langle \Upsilon
_{1}+\Upsilon _{3}+\Upsilon _{4}\rangle $ if $\alpha _{1}\alpha _{3}\neq
\alpha _{2}^{2}$ while we obtain the representative $\langle \Upsilon
_{3}+\Upsilon _{4}\rangle $ otherwise.

\item $\alpha _{3}=0$. Then we have the representative $\langle \Upsilon
_{1}+\Upsilon _{4}\rangle $ and $\langle \Upsilon _{2}+\Upsilon _{4}\rangle $
if $(\alpha _{1},\alpha _{2})\neq (0,0)$ while we get the representative $%
\langle \Upsilon _{4}\rangle $ otherwise.
\end{enumerate}

These representatives produce the algebras 
${\mathfrak P}_{06}$, ${\mathfrak P}_{07}$, ${\mathfrak P}_{08}$, and ${\mathfrak P}_{09},$ respectively.

\item[$\mathbb{P}_{2}:$]
Fix the following notation 
\begin{center}
$\Upsilon_1=\big[(\nabla_{11},\  0)\big], 
\Upsilon_2=\big[(\nabla_{12},\ 0)\big],
\Upsilon_3=\big[(\nabla_{22},\ 0)\big],
\Upsilon_4=\big[(\nabla_{23},\ 2\Delta_{13})\big],
\Upsilon_5=\big[(0,\ \Delta_{23})\big].$
\end{center}

Given $\Gamma =\sum_{i=1}^{5}\alpha _{i}\Upsilon _{i}$ in ${\rm H}^{2,2}\left( 
\mathbb{P}_{2},\mathbb{C}\right) $ such that $\mathrm{rad}(\Gamma )\cap 
\mathrm{Ann}(\mathbb{P}_{2})=0$. 
The automorphism group of $\mathbb{P}_{2}$, $\text{Aut}\left( \mathbb{P}_{2}\right) $, consists of
the automorphisms $\phi $ given by a matrix of the following form %
\[
\phi =%
\left(\begin{array}{ccc}
 x_{11} &  x_{12} & 0 \\
 0 &  x_{11} & 0 \\
 x_{31} &  x_{32} & x_{11}^2 \\
\end{array}\right).
\]%
Consider the action of an automorphism $%
\phi $, $\phi \Gamma =\sum_{i=1}^{5}\alpha _{i}^{\ast }\Upsilon _{i}$ where

\begin{longtable}{lcl}
$\alpha _{1}^{\ast } $&$=$&$\alpha _{1}x_{11}^{2},$ \\ 
$\alpha _{2}^{\ast } $&$=$&$x_{11}\left( \alpha _{1}x_{12}+\alpha _{2}x_{11}+\alpha_{4}x_{31}\right),$ \\ 
$\alpha _{3}^{\ast } $&$=$&$\alpha _{3}x_{11}^{2}+2\alpha _{2}x_{11}x_{12}+\alpha_{5}x_{31}x_{11}+\alpha _{1}x_{12}^{2}+2\alpha _{4} x_{31}x_{12},$ \\ 
$\alpha _{4}^{\ast }$&$=$&$\alpha _{4}x_{11}^{3},$ \\ 
$\alpha _{5}^{\ast }$&$=$&$x_{11}^{2}\left( 2\alpha _{4}x_{12}+\alpha_{5}x_{11}\right) .$ \\ 
\end{longtable}

Since $\mathrm{rad}(\Gamma )\cap \mathrm{Ann}(\mathbb{P}_{2})=0$, we have 
$\left( \alpha _{4},\alpha _{5}\right) \neq \left( 0,0\right) $. Let us
consider the following cases.

\begin{enumerate}
\item $\alpha _{4}=0$ and $\left( \alpha _{1},\alpha _{2}\right) = \left(
0,0\right).$ We get the
representative $W_{1}=\big\langle \Upsilon _{5}\big\rangle $.

\item $\alpha _{4}=0$ and $\left( \alpha _{1},\alpha _{2}\right) \neq \left(
0,0\right).$ We get the representative 
$W_{2}=\big\langle \Upsilon_{1}+\Upsilon _{5}\big\rangle $ if $\alpha _{1}\neq 0$ while we get the
representative $W_{3}=\big\langle \Upsilon _{2}+\Upsilon _{5}\big\rangle $ if $\alpha _{1}=0$.

\item $\alpha _{4}\neq 0$ and $\alpha _{1}=0.$  We get the representative 
$W_{4}=\big\langle \Upsilon _{4}\big\rangle $ when $\alpha _{3}=\alpha _{2}\alpha _{5}$ while
we get the representative $W_{5}=\big\langle \Upsilon _{3}+\Upsilon_{4}\big\rangle $ when $\alpha _{3}\neq \alpha _{2}\alpha _{5}$.

\item $\alpha _{4}\neq 0$ and $\alpha _{1}\neq 0.$ We get the representatives $W_{6}^{\alpha
}=\big\langle \Upsilon _{1}+\alpha \Upsilon _{3}+\Upsilon _{4}: \alpha\in 
\mathbb{C}\big\rangle $. 
Moreover, $W_{6}^{\alpha }$ and $W_{6}^{\beta }$ are in
the same orbit if and only if $\alpha =\beta $.
\end{enumerate}

From these representatives, we construct   
${\mathfrak P}_{10},$
${\mathfrak P}_{11},$ 
${\mathfrak P}_{12},$ 
${\mathfrak P}_{13},$ 
${\mathfrak P}_{14},$ and 
${\mathfrak P}_{15}^{\alpha}.$

\item[$\mathbb{P}_{3}^{\alpha}:$]
Fix the following notation  \begin{center}
$\Upsilon _{1}=\big[ (\nabla _{12},\ 0)\big],$  
$\Upsilon _{2}=\big[ (\alpha \left( \alpha -1\right) \nabla _{22}+\alpha \nabla_{13},\ \Delta _{13})\big].$
\end{center} 
Let $\Gamma =\alpha _{1}\Upsilon _{1}+\alpha
_{2}\Upsilon _{2}$ in ${\rm H}^{2,2}\left( \mathbb{P}_{3}^{\alpha },\mathbb{C}%
\right) $ such that $\mathrm{rad}(\Gamma )\cap \mathrm{Ann}(\mathbb{P}_{3}^{\alpha })=0$. 
The automorphism group of $\mathbb{P}_{3}^\alpha$, $\text{Aut}\left( \mathbb{P}_{3}^\alpha\right) $, consists of
the automorphisms $\phi $ given by a matrix of the following form %
\[
\phi =%
\left(\begin{array}{ccc}
 x_{11} &  0 & 0 \\
x_{21} &  x_{11}^2 & 0 \\
 x_{31} &  2x_{11}x_{21} & x_{11}^3 \\
\end{array}\right).
\]%
Consider the action of an automorphism $\phi $, then 
\begin{center}
    $\phi \Gamma =x_{11}^{2}\big( \allowbreak 3\alpha \left(\alpha -1\right) \alpha _{2}x_{21} +\alpha _{1}x_{11}\big) \Upsilon _{1}+\alpha
_{2}x_{11}^{4}\Upsilon _{2}$.
\end{center}
Since rad$(\Gamma )\cap $Ann$(\mathbb{P}_{3}^{\alpha})=0$, we have $\alpha _{2}\neq 0$. We distinguish two cases:

\begin{enumerate}
\item $\alpha _{1}=0$. In this case, we get the representative $%
W_{1}=\big\langle \Upsilon _{2}\big\rangle $.

\item $\alpha _{1}\neq 0$. We get the representative 
\begin{longtable}{lcl}
$W_{1}$&$=$&$\big\langle\Upsilon _{2}\big\rangle $ if $\alpha \left( \alpha -1\right) \neq 0,$\\
$W_{2}$&$=$&$\big\langle \Upsilon _{1}+\Upsilon_{2}\big\rangle$ if $\alpha \left( \alpha -1\right) =0.$
\end{longtable}
\end{enumerate}

Thus, we obtain the algebras ${\mathfrak P}_{16}^{\alpha},$ ${\mathfrak P}_{17}$ and ${\mathfrak P}_{18}$.

    \item [$\mathbb{P}_{4}:$]
 Fix the following notation 
\begin{center}
$\Upsilon_1=\big[(\nabla_{11},\  0)\big],$ \ 
$\Upsilon_2=\big[(\nabla_{12},\  0)\big],$  \ 
$\Upsilon_3=\big[(\nabla_{13},\  0)\big],$  \
$\Upsilon_4=\big[(\nabla_{22},\  0)\big],$ \ 
$\Upsilon_5=\big[(\nabla_{23},\  0)\big],$ \ 
$\Upsilon_6=\big[(\nabla_{33},\  0)\big],$ \ 
$\Upsilon_7=\big[(0,\  \Delta_{12})\big],$ \ 
$\Upsilon_8=\big[(0,\  \Delta_{13})\big],$  \ 
$\Upsilon_9=\big[(0,\  \Delta_{23})\big].$ \
\end{center}

Given $\Gamma = \sum_{i=1}^9\alpha_i \Upsilon_i$ in ${\rm H}^{2,2}\left( 
\mathbb{P}_{4},\mathbb{C}\right)$ such that $\text{rad}(\Gamma)=0$.
Denote $\Gamma =\left[ \left( \theta ,\vartheta \right) \right] $. Then $%
\theta $ is a symmetric bilinear form on $\mathbb{P}_{4}$ while $%
\vartheta $ is a skew-symmetric bilinear form on $\mathbb{P}_{4}$. As $%
\dim \mathbb{P}_{4}=3$, we have $\theta \neq 0$ since otherwise $\mathrm{%
rad}(\Gamma )\cap \mathrm{Ann}(\mathbb{P}_{4})\neq 0$. Since $\theta $ is
a symmetric bilinear form on $\mathbb{P}_{4}$, it is equivalent to a
diagonal one, and we may assume without any loss of generality that 
\begin{center}
    $\theta
\in \big\{ \nabla_{11},\  
\nabla_{11} + \nabla_{22},\ 
\nabla_{11}+\nabla_{22}+\nabla_{33}\big\} $ 
and 
$\big(\alpha_7, \alpha_8,\alpha_9\big) \neq \big(0,0,0\big).$
\end{center}
Then, we have the following
cases:

\begin{enumerate}
    \item[I.] $\theta =\nabla_{11}$. Then $\alpha _{9}\neq 0$ since
othewise $\mathrm{rad}(\Gamma )\cap \mathrm{Ann}(\mathbb{P}_{4})\neq 0$.
Now, let $\phi $ be the following automorphism:%
\begin{equation*}
\phi =\frac{1}{\alpha _{9}} 
\begin{pmatrix}
\alpha _{9} & 0 & 0 \\ 
-\alpha _{8} & \alpha _{9} & 0 \\ 
\alpha _{7} & 0 & 1%
\end{pmatrix}%
. 
\end{equation*}%
Then $\phi \Gamma =\Upsilon _{1}+\Upsilon _{9}$ and therefore we have the representative 
$W_{1}=\big\langle \Upsilon _{1}+\Upsilon
_{9}\big\rangle $.

\item[II.]
$\theta =\nabla_{11} + \nabla_{22}$. Then $\left( \alpha
_{8},\alpha _{9}\right) \neq \left( 0,0\right) $.

\begin{enumerate}[(a)]
\item $\alpha _{8}^{2}+\alpha _{9}^{2}\neq 0$. Let $\phi $ be the first of
the following matrices if $\alpha _{9}=0 $ or the second if $\alpha _{9}\neq
0$:%
\begin{equation*}
\begin{pmatrix}
1 & 0 & 0 \\ 
0 & -1 & 0 \\ 
0 & \alpha _{7}\alpha _{8}^{-1} & \alpha _{8}^{-1}%
\end{pmatrix}
\ \ \mbox{ or }\ \ \frac{1}{\sqrt{\alpha _{8}^{2}+\alpha _{9}^{2}}} 
\begin{pmatrix}
\alpha _{8} & \alpha _{9} & 0 \\ 
\alpha _{9} & -\alpha _{8} & 0 \\ 
0 & \alpha _{7} & 1%
\end{pmatrix}%
. 
\end{equation*}%
Then $\phi \Gamma =\Upsilon _{1}+\Upsilon _{4}+\Upsilon _{8}$. So we get the
representative $W_{2}=\big\langle \Upsilon _{1}+\Upsilon _{4}+\Upsilon_{8}\big\rangle $.

\item $\alpha _{8}^{2}+\alpha _{9}^{2}=0$. Let $\phi $ be the following
automorphism:%
\begin{equation*}
\phi =%
\begin{pmatrix}
{\bf i}{\alpha _{8}}\alpha^{-1}_{9} & 0 & 0 \\ 
0 & -{\bf i} & 0 \\ 
0 & -{\bf i}\alpha _{7}{\alpha _{8}}{\alpha _{9}^{-2}} & {\bf i}\alpha^{-1}_{9}%
\end{pmatrix}%
:\ {\bf i}=\sqrt{-1}. 
\end{equation*}%
Then $\phi \Gamma =\Upsilon _{1}-\Upsilon _{4}+\Upsilon _{8}+\Upsilon _{9}$.
Thus, we get  
    $W_{3}=\big\langle \Upsilon _{1}-\Upsilon_{4}+\Upsilon _{8}+\Upsilon _{9}\big\rangle $.
\end{enumerate}

\item[III.]
 $\theta =\nabla_{11}+\nabla_{22}+\nabla_{33}$. Then $%
\theta $ is a nondegenerate symmetric bilinear form on $\mathbb{P}_{4}$.
Up to equivalence, there is only one nondegenerate symmetric bilinear form on 
$\mathbb{P}_{4}$ and so we may assume that $\theta
=\nabla_{12}+\nabla_{33}$.

\begin{enumerate}[(a)]
\item $\lambda :=\alpha _{7}^{2}-2\alpha _{8}\alpha _{9}\neq 0$. Consider the
following matrices:%
\begin{center}
$\phi _{1} \ =\ 
\begin{pmatrix}
1 & 0 & 0 \\ 
0 & \alpha _{8}^{2} & 0 \\ 
0 & 0 & \alpha _{8}%
\end{pmatrix},$ \ 
$\phi _{2}=
\begin{pmatrix}
\frac{1}{\alpha _{7}} & -\frac{\alpha _{7}}{2} & 1 \\ 
-\frac{1}{\alpha _{7}} & \frac{\alpha _{7}}{2}  & 1 \\ 
\frac{\sqrt{2}}{\alpha _{7}} & \frac{\alpha _{7}}{\sqrt{2}} & 0%
\end{pmatrix},$ \ 
$\phi _{3}\ = \ 
\begin{pmatrix}
-1 & \frac{\alpha _{7}^{2}}{2} & -\alpha _{7} \\ 
\frac{2\alpha _{8}^{2}}{\alpha _{7}^{2}} & 0 & 0 \\ 
-\frac{2\alpha _{8}}{\alpha _{7}} & 0 & -\alpha _{8}%
\end{pmatrix}
,$ \
$\phi _{4} \ =\ 
\begin{pmatrix}
\frac{2\alpha _{9}^{2}}{\alpha _{7}^{2}} & 0 & 0 \\ 
-1 & \frac{\alpha _{7}^{2}}{2} & \alpha _{7} \\ 
\frac{2\alpha _{9}}{\alpha _{7}} & 0 & -\alpha _{9}%
\end{pmatrix}%
,$ \ 
$\phi _{5}\ = \ 
\begin{pmatrix}
\frac{\alpha _{9}}{\alpha _{7}} & 0 & 0 \\ 
-\frac{\alpha _{7}}{2{\alpha _{9}}} & \frac{\lambda ^{2}}{4\alpha _{7}\alpha
_{9}} &  \frac{\lambda }{2\alpha _{9}} \\ 
1 & 0 & -\frac{\lambda}{2\alpha _{7}} %
\end{pmatrix}%
.$
\end{center}

\begin{enumerate}
\item If $\alpha _{7}=0$, we choose $\phi =\phi _{1}$ and we get 
$W_{4}^{\alpha \neq 0}=\big\langle \Upsilon _{2}+\Upsilon
_{6}+\Upsilon _{8}+\alpha \Upsilon _{9} \big\rangle $.

\item If $\alpha _{7}\neq 0$ and $\alpha _{8}=\alpha _{9}=0$, we choose $%
\phi =\phi _{2}$ and     get   $W_{4}^{\alpha \neq 0}.$

\item If $\alpha _{7}\alpha _{8}\neq 0$ and $\alpha _{9}=0$, we choose $\phi
=\phi _{3}$ and     get  $W_{4}^{\alpha \neq 0}$.

\item If $\alpha _{7}\alpha _{9}\neq 0$ and $\alpha _{8}=0$, we choose $\phi
=\phi _{4}$ and  get  $W_{4}^{\alpha \neq 0}$.

\item If $\alpha _{7}\alpha _{8}\alpha _{9}\neq 0$, we choose $\phi =\phi
_{5}$ and  get    $W_{4}^{\alpha \neq 0}$.
\end{enumerate}

\item $\lambda :=\alpha _{7}^{2}-2\alpha _{8}\alpha _{9}=0,$ then $\left( \alpha _{8},\alpha _{9}\right) \neq \left( 0,0\right) $, we
choose $\phi $\ to be the first of the following matrices when $\alpha
_{9}\neq 0$ or the second otherwise:%
\begin{equation*}
\begin{pmatrix}
0 & \alpha _{9}^{2} & 0 \\ 
1 & -\alpha _{8}\alpha _{9} & -\alpha _{7} \\ 
0 & \alpha _{7}\alpha _{9} & \alpha _{9}%
\end{pmatrix}%
\mbox{ or }
\begin{pmatrix}
1 & 0 & 0 \\ 
0 & \alpha _{8}^{2} & 0 \\ 
0 & 0 & \alpha _{8}%
\end{pmatrix}%
. 
\end{equation*}%
Then we get the representative 
$W_4^{0}=\big\langle \Upsilon_{2}+\Upsilon _{6}+\Upsilon _{8}\big\rangle $.
 
\end{enumerate}

Thus, in case of $\theta =\Upsilon _{2}+\Upsilon _{6}$, we have the
representatives $W_{4}^{\alpha }$. 
Moreover, $W_{4}^{\alpha }$ and $W_{4}^{\beta }$ are in the same orbit if and only if $\alpha =\beta $.

The annihilator extensions of $\mathbb{P}_{4}$ correspond to the algebras 
${\mathfrak P}_{19}$, ${\mathfrak P}_{20}$, ${\mathfrak P}_{21}$, and ${\mathfrak P}_{22}^{\alpha }$, 
respectively, constructed from the representatives of orbits obtained in this section.

\end{enumerate}

\item[$\mathbb{P}_{5}:$]
Fix the following notation 
\begin{center}
$\Upsilon_1=\big[(\nabla_{12},\  0)\big],$ \
$\Upsilon_2=\big[(\nabla_{13},\  0)\big],$ \ 
$\Upsilon_3=\big[(\nabla_{33},\  0)\big],$ \ 
$\Upsilon_4=\big[(0,\  \Delta_{12})\big],$\ 
$\Upsilon_5=\big[(0,\  \Delta_{13})\big].$  
\end{center}

Given $\Gamma =\sum_{i=1}^{5}\alpha _{i}\Upsilon _{i}$ in ${\rm H}%
^{2,2}\left( \mathbb{P}_{5},\mathbb{C}\right) $ such that $\text{rad}%
(\Gamma )\cap \text{Ann}(\mathbb{P}_{5})=0$.
The automorphism group of $\mathbb{P}_{5}$, $\text{Aut}\left( \mathbb{P}_{5}\right) $, consists of
the automorphisms $\phi $ given by a matrix of the following form %
\[
\phi =%
\left(\begin{array}{ccc}
 x_{11} &  0 & 0 \\
 x_{21} &  x_{11}^2 & x_{23} \\
 x_{31} &  0 & x_{33} \\
\end{array}\right).
\]%
Consider the action of an
automorphism $\phi $, $\phi \Gamma =\sum_{i=1}^{5}\alpha _{i}^{\ast
}\Upsilon _{i}$ where

\begin{longtable}{lcllcllcl}
$\alpha _{1}^{\ast }$&$=$&$\alpha _{1}x_{11}^{3},$ &
$\alpha _{2}^{\ast }$&$=$&$\alpha _{1}x_{11}x_{23}+\alpha _{2}x_{11}x_{33}+\alpha_{3}x_{31}x_{33},$ &
$\alpha _{3}^{\ast }$&$=$&$\alpha _{3}x_{33}^{2},$ \\ 
$\alpha _{4}^{\ast }$&$=$&$\alpha _{4}x_{11}^{3},$ &
$\alpha _{5}^{\ast }$&$=$&$\alpha _{4}x_{11}x_{23}+\alpha _{5}x_{11}x_{33}.$  
\end{longtable}
 
\begin{enumerate}[I.]
    \item 
Suppose first that $\alpha _{4}=0,$ then $\alpha _{5}\neq 0$ and  we have the following cases:

\begin{enumerate}[(a)]
\item $\alpha _{3}\neq 0$. Then we have the representative 
$\big\langle \Upsilon_{1}+\Upsilon _{3}+\Upsilon _{5}\big\rangle $ if $\alpha _{5}\neq 0$, 
denote its algebra by ${\mathfrak P}_{23}$. 

\item $\alpha _{3}=0.$ 
We obtain the representative $\big\langle \Upsilon _{1}+\Upsilon _{5}\big\rangle $
and the algebra ${\mathfrak P}_{24}$.
\end{enumerate}

\item Assume now that $\alpha _{4}\neq 0$. Then, we have the following cases:

\begin{enumerate}[(a)]
\item $\alpha _{3}\neq 0$. Then we have the representatives 
$\big\langle \alpha\Upsilon _{1}+\Upsilon _{3}+\Upsilon _{4}\big\rangle $. 
So we get the algebras ${\mathfrak P}_{25}^{\alpha }$. Moreover, ${\mathfrak P}_{25}^{\alpha }$ and ${\mathfrak P}_{25}^{\beta }$ are  
isomorphic    if and only if $\alpha =\beta $.

\item $\alpha _{3}=0$. Then we have the representatives 
$\big\langle \alpha \Upsilon _{1}+\Upsilon _{2}+\Upsilon _{4}\big\rangle $. 
So we get the algebras ${\mathfrak P}_{26}^{\alpha }$. Furthermore, ${\mathfrak P}_{26}^{\alpha }$
and   ${\mathfrak P}_{26}^{\beta }$ are isomorphic if and only if $\alpha =\beta 
$.
\end{enumerate}
 
\end{enumerate}

\item [$\mathbb{P}_{7}:$]
Fix the following notation  
\begin{center}$\Upsilon _{1}=\big[ (\nabla _{13}+\nabla _{22},\ 0)\big]$ and  
$\Upsilon _{2}= \big[ (0,\ \Delta _{12})\big]$.\end{center} 
Let $\Gamma =\alpha_{1}\Upsilon _{1}+\alpha _{2}\Upsilon _{2}$ in ${\rm H}^{2,2}\left( \mathbb{P}_{7},\mathbb{C}\right),$ 
such that $\mathrm{rad}(\Gamma )\cap \mathrm{Ann}%
(\mathbb{P}_{7})=0$ and $\alpha_2\neq0.$ 
The automorphism group of $\mathbb{P}_{7}$, $\text{Aut}\left( \mathbb{P}_{7}\right) $, consists of
the automorphisms $\phi $ given by a matrix of the following form %
\[
\phi =%
\left(\begin{array}{ccc}
 x_{11} &  0 & 0 \\
x_{21} &  x_{11}^2 & 0 \\
 x_{31} &  2x_{11}x_{21} & x_{11}^3 \\
\end{array}\right).
\]%
Consider the action of an automorphism $\phi $, then 
$\phi \Gamma =\alpha _{1}x_{11}^{4}\Upsilon _{1}+\alpha
_{2}x_{11}^{3}\Upsilon _{2}$. Since $\mathrm{rad}(\Gamma )\cap \mathrm{Ann}(%
\mathbb{P}_{7})=0$, we have $\alpha _{1}\neq 0$. So we obtain the representative 
$\big\langle \Upsilon _{1}+\Upsilon_{2}\big\rangle$. 
These representative correspond
to algebra ${\mathfrak P}_{27}$.

\end{enumerate}

\section{The geometric classification of   CPAs}

 \subsection{Preliminaries: degenerations and    geometric classification}
 Given a complex vector space ${\mathbb V}$ of dimension $n$, the set of bilinear maps \begin{center}$\textrm{Bil}({\mathbb V} \times {\mathbb V}, {\mathbb V}) \cong \textrm{Hom}({\mathbb V} ^{\otimes2}, {\mathbb V})\cong ({\mathbb V}^*)^{\otimes2} \otimes {\mathbb V}$
\end{center} is a vector space of dimension $n^3$. The set of pairs of bilinear maps 
\begin{center}$\textrm{Bil}({\mathbb V} \times {\mathbb V}, {\mathbb V}) \oplus \textrm{Bil}({\mathbb V}\times {\mathbb V}, {\mathbb V}) \cong \big(({\mathbb V}^*)^{\otimes2} \otimes {\mathbb V}\big) \oplus \big(({\mathbb V}^*)^{\otimes2} \otimes {\mathbb V}\big)$ \end{center} which is a vector space of dimension $2n^3$. This vector space has the structure of the affine space $\mathbb{C}^{2n^3}$ in the following sense:
fixed a basis $e_1, \ldots, e_n$ of ${\mathbb V}$, then any pair $(\mu, \mu') \in \textrm{Bil}({\mathbb V} \times {\mathbb V}, {\mathbb V})^2$ is determined by some parameters $c_{ij}^k, d_{ij}^k \in \mathbb{C}$,  called {structural constants},  such that
$$\mu(e_i, e_j) = \sum_{p=1}^n c_{ij}^k e_k \textrm{ and } \mu'(e_i, e_j) = \sum_{p=1}^n d_{ij}^k e_k$$
which corresponds to a point in the affine space $\mathbb{C}^{2n^3}$. Then a subset $\mathcal S$ of $\textrm{Bil}({\mathbb V} \times {\mathbb V}, {\mathbb V})^2$ corresponds to an algebraic variety, i.e., a Zariski closed set, if there are some polynomial equations in variables $c_{ij}^k, d_{ij}^k$ with zero locus equal to the set of structural constants of the bilinear pairs in $\mathcal S$. 

Given the identities defining {\rm CPA}s, we can obtain a set of polynomial equations in the variables $c_{ij}^k$ and $d_{ij}^k$. This class of $n$-dimensional {\rm CPA}s is a variety. Denote it by $\mathcal{T}_{n}$.
Now, consider the following action of $\rm{GL}({\mathbb V})$ on ${\mathcal T}_{n}$:
$$\big(g*(\mu, \mu')\big)(x,y) := \big(g \mu (g^{-1} x, g^{-1} y), g \mu' (g^{-1} x, g^{-1} y)\big)$$
for $g\in\rm{GL}({\mathbb V})$, $(\mu, \mu')\in \mathcal{T}_{n}$ and for any $x, y \in {\mathbb V}$. Observe that the $\textrm{GL}({\mathbb V})$-orbit of $(\mu, \mu')$, denoted ${\mathcal O}\big((\mu, \mu')\big)$, contains all the structural constants of the bilinear pairs isomorphic to the {\rm CPA} with structural constants $(\mu, \mu')$.

A geometric classification of a variety of algebras consists of describing the irreducible components of the variety. Recall that any affine variety can be represented uniquely as a finite union of its irreducible components.
Note that describing the irreducible components of  ${\mathcal{T}_{n}}$ gives us the rigid algebras of the variety, which are those bilinear pairs with an open $\textrm{GL}(\mathbb V)$-orbit. This is because a bilinear pair is rigid in a variety if and only if the closure of its orbit is an irreducible component of the variety. 
For this reason, the following notion is convenient. Denote by $\overline{{\mathcal O}\big((\mu, \mu')\big)}$ the closure of the orbit of $(\mu, \mu')\in{\mathcal{T}_{n}}$.

\begin{definition}
\rm Let ${\rm T} $ and ${\rm T}'$ be two $n$-dimensional CPAs  corresponding to the variety $\mathcal{T}_{n}$ and $(\mu, \mu'), (\lambda,\lambda') \in \mathcal{T}_{n}$ be their representatives in the affine space, respectively. The algebra ${\rm T}$ is said to {degenerate}  to ${\rm T}'$, and we write ${\rm T} \to {\rm T} '$, if $(\lambda,\lambda')\in\overline{{\mathcal O}\big((\mu, \mu')\big)}$. If ${\rm T}  \not\cong {\rm T}'$, then we call it a  {proper degeneration}.
Conversely, if $(\lambda,\lambda')\not\in\overline{{\mathcal O}\big((\mu, \mu')\big)}$ then  we say that ${{\rm T} }$ does not degenerate to ${{\rm T} }'$
and we write ${{\rm T} }\not\to {{\rm T} }'$.
\end{definition}

Furthermore, we have the following notion for a parametric family of algebras.

\begin{definition}
\rm
Let ${{\rm T} }(*)=\{{{\rm T} }(\alpha): {\alpha\in I}\}$ be a family of $n$-dimensional {\rm CPA}s  corresponding to ${\mathcal{T} }_n$ and let ${{\rm T} }'$ be another $n$-dimensional {\rm CPA}. Suppose that ${{\rm T} }(\alpha)$ is represented by the structure $\big(\mu(\alpha),\mu'(\alpha)\big)\in{\mathcal{T} }_n$ for $\alpha\in I$ and ${{\rm T} }'$ is represented by the structure $(\lambda, \lambda')\in{\mathcal{T} }_n$. We say that the family ${{\rm T} }(*)$ {degenerates}   to ${{\rm T} }'$, and write ${{\rm T} }(*)\to {{\rm T} }'$, if $(\lambda,\lambda')\in\overline{\{{\mathcal O}\big((\mu(\alpha),\mu'(\alpha))\big)\}_{\alpha\in I}}$.
Conversely, if $(\lambda,\lambda')\not\in\overline{\{{\mathcal O}\big((\mu(\alpha),\mu'(\alpha))\big)\}_{\alpha\in I}}$ then we call it a  {non-degeneration}, and we write ${{\rm T} }(*)\not\to {{\rm T} }'$.

\end{definition}

Observe that ${\rm T}'$ corresponds to an irreducible component of $\mathcal{T}_n$ (more precisely, $\overline{{\rm T}'}$ is an irreducible component) if and only if ${{\rm T} }\not\to {{\rm T} }'$ for any $n$-dimensional CPA ${\rm T}$ and ${{\rm T}(*) }\not\to {{\rm T} }'$ for any parametric family of $n$-dimensional CPAs ${\rm T}(*)$. In this case, we will use the following ideas to prove that a particular algebra corresponds to an irreducible component.
Firstly, since $\mathrm{dim}\,{\mathcal O}\big((\mu, \mu')\big) = n^2 - \mathrm{dim}\,\mathfrak{Der}({\rm T})$, then if $ {\rm T} \to  {\rm T} '$ and  ${\rm T} \not\cong  {\rm T} '$, we have that $\mathrm{dim}\,\mathfrak{Der}( {\rm T} )<\mathrm{dim}\,\mathfrak{Der}( {\rm T} ')$, where $\mathfrak{Der}( {\rm T} )$ denotes the Lie algebra of derivations of  ${\rm T} $. 
Secondly, to prove degenerations, let ${{\rm T} }$ and ${{\rm T} }'$ be two {\rm CPA}s represented by the structures $(\mu, \mu')$ and $(\lambda, \lambda')$ from ${{\mathcal T} }_n$, respectively. Let $c_{ij}^k, d_{ij}^k$ be the structure constants of $(\lambda, \lambda')$ in a basis $e_1,\dots, e_n$ of ${\mathbb V}$. If there exist $n^2$ maps $a_i^j(t): \mathbb{C}^*\to \mathbb{C}$ such that $E_i(t)=\sum_{j=1}^na_i^j(t)e_j$ ($1\leq i \leq n$) form a basis of ${\mathbb V}$ for any $t\in\mathbb{C}^*$ and the structure constants $c_{ij}^k(t), d_{ij}^k(t)$ of $(\mu, \mu')$ in the basis $E_1(t),\dots, E_n(t)$ satisfy $\lim\limits_{t\to 0}c_{ij}^k(t)=c_{ij}^k$ and $\lim\limits_{t\to 0}d_{ij}^k(t)=d_{ij}^k$, then ${{\rm T} }\to {{\rm T} }'$. In this case,  $E_1(t),\dots, E_n(t)$ is called a parametrized basis for ${{\rm T} }\to {{\rm T} }'$
and denote it as ${\rm T} \xrightarrow{E_1(t),\dots, E_n(t)}  {{\rm T} }'.$
Thirdly, to prove non-degenerations, we may use a remark that follows from this lemma, see \cite{afm}  and the references therein. 

\begin{lemma} 
Consider two {\rm CPA}s ${\rm T}$ and ${\rm T}'$. Suppose ${\rm T} \to {\rm T}'$. Let ${\rm C}$   be a Zariski closed set in ${\mathcal T}_n$ that is stable by the action of the invertible upper (lower) triangular matrices. Then if there is a representation $(\mu, \mu')$ of ${\rm T}$ in ${\rm C}$, then there is a representation $(\lambda, \lambda')$ of ${\rm T}'$ in ${\rm C}$.
\end{lemma}

To apply this lemma, we will give the explicit definition of the appropriate stable Zariski closed ${\rm C}$ in terms of the variables $c_{ij}^k, d_{ij}^k$ in each case. For clarity, we assume by convention that $c_{ij}^k=0$ (resp. $d_{ij}^k=0$) if $c_{ij}^k$ (resp. $d_{ij}^k$) is not explicitly mentioned on the definition of ${\rm C}$.

\begin{remark}
\label{redbil}

Moreover, let ${{\rm T} }$ and ${{\rm T} }'$ be two {\rm CPA}s represented by the structures $(\mu, \mu')$ and $(\lambda, \lambda')$ from ${{\rm T} }_n$. Suppose ${\rm T}\to{\rm T}'$. Then if $\mu, \mu', \lambda, \lambda'$ represents algebras ${\rm T}_{0}, {\rm T}_{1}, {\rm T}'_{0}, {\rm T}'_{1}$ in the affine space $\mathbb{C}^{n^3}$ of algebras with a single multiplication, respectively, we have ${\rm T}_{0}\to {\rm T}'_{0}$ and $ {\rm T}_{1}\to {\rm T}'_{1}$.
So for example, $(0, \mu)$ can not degenerate in $(\lambda, 0)$ unless $\lambda=0$. 

\end{remark}

Fourthly, to prove ${{\rm T} }(*)\to {{\rm T} }'$, suppose that ${{\rm T} }(\alpha)$ is represented by the structure $(\mu(\alpha),\mu'(\alpha))\in{\mathcal{T} }_n$ for $\alpha\in I$ and ${{\rm T} }'$ is represented by the structure $(\lambda, \lambda')\in{\mathcal{T} }_n$. Let $c_{ij}^k, d_{ij}^k$ be the structure constants of $(\lambda, \lambda')$ in a basis  $e_1,\dots, e_n$ of ${\mathbb V}$. If there is a pair of maps $(f, (a_i^j))$, where $f:\mathbb{C}^*\to I$ and $a_i^j:\mathbb{C}^*\to \mathbb{C}$ are such that $E_i(t)=\sum_{j=1}^na_i^j(t)e_j$ ($1\le i\le n$) form a basis of ${\mathbb V}$ for any  $t\in\mathbb{C}^*$ and the structure constants $c_{ij}^k(t), d_{ij}^k(t)$ of $(\mu\big(f(t)\big),\mu'\big(f(t)\big))$ in the basis $E_1(t),\dots, E_n(t)$ satisfy $\lim\limits_{t\to 0}c_{ij}^k(t)=c_{ij}^k$ and $\lim\limits_{t\to 0}d_{ij}^k(t)=d_{ij}^k$, then ${{\rm T} }(*)\to {{\rm T} }'$. In this case  $E_1(t),\dots, E_n(t)$ and $f(t)$ are called a parametrized basis and a parametrized index for ${{\rm T} }(*)\to {{\rm T} }'$, respectively.
Fifthly, to prove ${{\rm T} }(*)\not \to {{\rm T} }'$, we may use an analogous of Remark \ref{redbil} for parametric families that follows from Lemma \ref{main2}, see \cite{afm}.

\begin{lemma}\label{main2}
Consider the family of {\rm CPA}s ${\rm T}(*)$ and the {\rm CPA} ${\rm T}'$. Suppose ${\rm T}(*) \to {\rm T}'$. Let ${\rm C}$  be a Zariski closed subset in $\mathcal{T}_n$ that is stable by the action of the invertible upper (lower) triangular matrices. Then if there is a representation $(\mu(\alpha), \mu'(\alpha))$ of ${\rm T}(\alpha)$ in ${\rm C}$ for every $\alpha\in I$, then there is a representation $(\lambda, \lambda')$ of ${\rm T}'$ in ${\rm C}$.
\end{lemma}

Finally, the following remark simplifies the geometric problem.

\begin{remark}
 Let $(\mu, \mu')$ and $(\lambda, \lambda')$ represent two {\rm CPA}s. Suppose $(\lambda, 0)\not\in\overline{{\mathcal O}((\mu, 0))}$, $\big($resp., $(0, \lambda')\not\in\overline{{\mathcal O}((0, \mu'))}\big),$ then $(\lambda, \lambda')\not\in\overline{{\mathcal O}\big((\mu, \mu')\big)}$.
  As we construct the classification of {\rm CPA}s from a certain class of algebras with a single multiplication which remains unchanged, this remark becomes very useful.  
\end{remark}

\subsection{The geometric classification of $3$-dimensional CPA }

\begin{proposition}\label{rig}
    Let ${\rm P}$ be a simple $n$-dimensional Lie algebra, 
    then it is rigid in the variety of $n$-dimensional commutative post-Lie algebras.
\end{proposition}

 \begin{proof} 
   Thanks to Corollary \ref{cor}, there are no nontrivial $n$-dimensional commutative post-Lie algebras with Lie part isomorphic to a simple Lie algebra  ${\rm P}.$ 
   Hence,  ${\rm P}$ does not admit nontrivial deformations in   the variety of $n$-dimensional commutative post-Lie algebras.
  \end{proof}

\begin{proposition}\label{3gnil}
The variety of $3$-dimensional nilpotent commutative post-Lie
algebras is irreducible and defined by a $7$-dimensional component $\overline{\mathcal{O}(\mathbb{P}_{3}^{\alpha})}.$
\end{proposition}

 \begin{proof}
  All necessary degenerations are given below:

\begin{longtable}{lcl|lcl|lcl} \hline

$\mathbb{P}_{3}^{\alpha}$ & $\xrightarrow{ \big(te_1,\  e_2,\  te_3\big)}$ & $\mathbb{P}_{1}^{\alpha}$ & 
$\mathbb{P}_{3}^{t}$ & $\xrightarrow{ \big(e_2+t^{-1}e_3, \ te_1,\  -te_3\big)}$ & $\mathbb{P}_{2}$ & 
$\mathbb{P}_{3}^{0}$ & $\xrightarrow{ \big(te_1,\  e_2,\  e_3\big)}$ & $\mathbb{P}_{4}$ 
\\  \hline

$\mathbb{P}_{3}^{0}$ & $\xrightarrow{ \big(te_1,\  t^2e_2, \ e_3\big)}$ & $\mathbb{P}_{5}$ &
$\mathbb{P}_{3}^{t^{-1}}$ & $\xrightarrow{ \big(te_1, \ e_2,\  e_3\big)}$ & $\mathbb{P}_{6}$ & 
$\mathbb{P}_{3}^{t^{-3}}$ & $\xrightarrow{ \big(te_1,\ t^2e_2, \  e_3\big)}$ & $\mathbb{P}_{7}$

\\  \hline

\end{longtable}
 \end{proof}

\begin{theoremG1}
The variety of $3$-dimensional commutative post-Lie
algebras has  dimension  $8$ and it has  $9$  irreducible components defined by  
\begin{center}
$\mathcal{C}_1=\overline{\mathcal{O}( \mathbb{T}_{1})},$ \
$\mathcal{C}_2=\overline{\mathcal{O}(\mathfrak{L}_{1})},$ \ 
$\mathcal{C}_3=\overline{\mathcal{O}({\rm P}_{02}^{\alpha})},$ \
$\mathcal{C}_4=\overline{\mathcal{O}({\rm P}_{10})},$\
$\mathcal{C}_5=\overline{\mathcal{O}({\rm P}_{11})},$\
$\mathcal{C}_6=\overline{\mathcal{O}({\rm P}_{13}^{\alpha})},$\
$\mathcal{C}_7=\overline{\mathcal{O}({\rm P}_{14}^{\alpha})},$\
and 
$\mathcal{C}_8=\overline{\mathcal{O}({\rm P}_{16})}.$\
\end{center}
In particular, there are   $5$ rigid algebras.
\end{theoremG1}
\begin{proof}

A {\rm CPA} with one trivial multiplication is
    isomorphic to a commutative associative algebra or a Lie algebra. Thanks to \cite{MS}, the variety of $3$-dimensional  commutative associative  algebras has one irreducible component: 
\begin{longtable}{lllllll}
$\mathbb{T}_{1}$ & $:$ & $e_1\cdot e_1 =e_1$ &   $e_2\cdot e_2 =e_2$ & $e_3\cdot e_3 =e_3$  \\ 

\end{longtable}
    \noindent     
    and  the variety of $3$-dimensional Lie  algebras has two irreducible components defined by 

\begin{longtable}{lllllll}
$\mathfrak{L}_{1}$ & $:$ & $ \{e_1, e_2\} =e_3$ & $ \{e_1, e_3\} =-2e_1$ & $ \{e_2, e_3\} =2e_2$\\

$\mathfrak{L}_{2}^{\alpha}$ & $:$ & $ \{e_1, e_2\} =e_2$ & $ \{e_1, e_3\} =\alpha e_3$ \\
\end{longtable}
\noindent 
We now present some degenerations that demonstrate non-rigid algebras.

\begin{longtable}{lcl|lcl} \hline

${\rm P}_{05}^{\alpha}$ & $\xrightarrow{ \big(e_1,\  t^{-1}e_2, \ e_3\big)}$ & 
$\mathfrak{L}_{2}^{\alpha}$ & ${\rm P}_{10}$ & $\xrightarrow{ \big(e_1,\ t^{-1}e_3,\ -e_2\big)}$ & ${\rm P}_{01}$

\\  \hline

${\rm P}_{02}^0$ & $\xrightarrow{ \big(te_1,\ t^2e_2,\ t^3e_3\big)}$ & ${\rm P}_{03}$ & 
${\rm P}_{02}^{t^{-2}}$ & $\xrightarrow{ \big(te_1,\ t^2e_2,\ t^3e_3\big)}$ & ${\rm P}_{04}$

\\  \hline

\multicolumn{5}{@{}l@{}}{
\begin{tabular}{ll}     

${\rm P}_{02}^{-\frac{\alpha(1+t)}{(1+\alpha+t)^2}}$\quad \quad\quad & $\xrightarrow{ \big((1+\alpha+t)e_1, \ \frac{(1+t)(1+\alpha+t)^2}{1-\alpha+t}e_2+\frac{(1+\alpha+t)^3}{1-\alpha+t}e_3,\  \frac{\alpha(1+\alpha+t)^2}{t(\alpha-1-t)}e_2-\frac{(1+\alpha+t)^3}{t(1-\alpha+t)}e_3\big)}$  \quad\quad \end{tabular}} & ${\rm P}_{05}^{\alpha}$

\\  \hline

${\rm P}_{05}^{\frac{1}{1+t}}$ & $\xrightarrow{ \big((1+t)e_1,\ -te_3,\ (2+t+t^{-1})e_2+e_3\big)}$ & 
${\rm P}_{06}$ & ${\rm P}_{02}^t$ & $\xrightarrow{ \big(te_1, \ t^2e_2,\ e_3\big)}$ & ${\rm P}_{07}$

\\  \hline

${\rm P}_{09}^{\alpha}$ & $\xrightarrow{ \big(te_1,\ e_2,\ te_3 \big)}$ & ${\rm P}_{08}^{\alpha}$ &
${\rm P}_{10}$ & $\xrightarrow{ \big(te_1-\alpha te_2+\alpha^2te_3,\ \alpha t^2e_2,\ -\alpha^2t^3e_3\big)}$ & ${\rm P}_{09}^{\alpha\neq0}$ 

\\  \hline

${\rm P}_{14}^{\alpha}$ & $\xrightarrow{ \big(e_1,\ t^{-1}e_2,\ e_3\big)}$ & ${\rm P}_{12}^{\alpha}$ &

${\rm P}_{13}^{1+t}$ & $\xrightarrow{ \big(e_1+(1+t)e_2,\ te_2,\ \frac{1}{1+t}e_3\big)}$ & ${\rm P}_{15}$

\\  \hline

\end{longtable}

\noindent
By direct calculations,  we have 

\begin{longtable}{rclclcl}
      
&&&&$\dim  \mathcal{O}(\mathbb{T}_{1})$&$=$&$9,$ \\ 
$\dim  \mathcal{O}({\rm P}_{02}^{\alpha})\ =\ 
\dim  \mathcal{O}({\rm P}_{10})\ =\ 
\dim  \mathcal{O}({\rm P}_{11}) $&$ =$&$ 
\dim  \mathcal{O}({\rm P}_{13}^{\alpha})$&$=$&$
\dim  \mathcal{O}({\rm P}_{14}^{\alpha})
$&$=$&$8,$ \\
&&&&$\dim  \mathcal{O}({\rm P}_{16})
$&$=$&$7,$ \\
&&&&$\dim  \mathcal{O}(\mathfrak{L}_{1})$&$=$&$6.$ \\

\end{longtable}

We will now list all the important reasons for non-degenerations.

\begin{enumerate}
    \item   $(\mathfrak{L}_{1},\{-,-\})$  is rigid due to Proposition \ref{rig}.


        \item   
     ${\rm dim} \{{\rm P},{\rm P} \} =1,$ if  $ {\rm P} \in \big\{  {\rm P}_{10},{\rm P}_{11}  \big\},$ so 
     $\big\{ {\rm P}_{10},{\rm P}_{11}  \big\}\not\to \big\{      {\rm P}_{02}^{\alpha}, \ {\rm P}_{13}^{\alpha},\ {\rm P}_{14}^{\alpha}, \ {\rm P}_{16}^{}\big\}.$


     \item    $({\rm P}_{02}^{\alpha},\cdot)$  is nilpotent, so 
     $  {\rm P}_{02}^{\alpha} \not\to \big\{    {\rm P}_{10}, {\rm P}_{11} \big\}.$ 



      \item The last reasons for non-degenerations are given below

      \begin{longtable}{lcl|l} 
      \hline 
      ${\rm P}_{13}^{\alpha\neq0}$ & $\not \rightarrow  $ & ${\rm P}_{17}$ &  
      $\mathcal R=\left\{\begin{array}{lll} 
      c_{22}^1=c_{23}^1=0, & c_{33}^1=c_{33}^2=0, & c_{12}^1+d_{12}^1=0, \\
      c_{13}^1+d_{13}^1=0, & c_{13}^3+d_{13}^3=0, & c_{23}^2+d_{23}^2=0.
      \end{array}\right.$\\ 
      \hline 

      ${\rm P}_{14}^{\alpha}$ & $\not \rightarrow  $ & ${\rm P}_{17}$ &  
      $\mathcal R=\left\{\begin{array}{ll} 
      c_{22}^1=c_{23}^1=c_{33}^1=0, &  c_{12}^1+d_{12}^1=0, \\
      c_{11}^1+2c_{12}^2+2c_{13}^3=0,  &  c_{13}^1+d_{13}^1=0.
      \end{array}\right.$\\ 
      \hline 
      
      \end{longtable}
 \end{enumerate}

\end{proof}

\subsection{The geometric classification of $4$-dimensional nilpotent CPA   }

\begin{theoremG2} The variety of $4$-dimensional nilpotent commutative post-Lie
algebras has  dimension  $14$ and it has  $2$  irreducible components defined by  
$\mathcal{C}_1=\overline{\mathcal{O}( \mathfrak{P}_{15}^\alpha)}$  and
$\mathcal{C}_2=\overline{\mathcal{O}(\mathfrak{P}_{16}^\alpha)}.$
In particular, there are no rigid algebras in this variety.
\end{theoremG2}
\begin{proof}
A {\rm CPA} with one trivial multiplication is
    isomorphic to a commutative associative algebra or a Lie algebra. Thanks to \cite{MS}, the variety of $4$-dimensional nilpotent commutative associative  algebras has one irreducible component given by
    \begin{longtable}{lllllll}
$\mathfrak{U}_{1}$ & $:$ & $e_1\cdot e_1 =e_2$ &   $e_1\cdot e_2 =e_3$  &   $e_1\cdot e_3 =e_4$ &   $e_2\cdot e_2 =e_4$\end{longtable}
\noindent
and  $4$-dimensional nilpotent Lie  algebras have one irreducible component defined by 
\begin{longtable}{lllllll}

$\mathcal{A}_{2}$ & $:$ & $ \{e_1, e_2\} =e_3$ & $ \{e_1, e_3\} =e_4.$ \\

\end{longtable}
\noindent Thanks to Proposition \ref{3gnil}, all split algebras are degenerated from $\mathbb{P}_{3}^{\alpha}.$

\begin{longtable}{lcl} \hline

$\mathfrak{P}_{16}^{t^{-1}}$ & $\xrightarrow{ \big(te_1,\  t^2e_2,\  t^2e_3,\ t^2e_4 \big)}$ & $\mathfrak{U}_{1}$ 
\\  \hline
$\mathfrak{P}_{16}^{0}$ & $\xrightarrow{ \big(te_1,\  e_2,\ te_3,\ t^2e_4\big)}$ & $\mathcal{A}_{2}$  
\\  \hline

$\mathfrak{P}_{16}^{\alpha}$ & $\xrightarrow{ \big(te_1,\  t^2e_2,\  t^3e_3,\ e_4\big)}$ & $\mathbb{P}_{3}^{\alpha} $ \\  \hline

$\mathfrak{P}_{04}^{-(\alpha+t)^2}$ & $\xrightarrow{ \big(-(\alpha+t)e_1+e_2,\  te_2,\  -t(\alpha+t)e_3,\ -2(\alpha+t)e_4 \big)}$ & $\mathfrak{P}_{01}^{\alpha}$ 
\\  \hline

$\mathfrak{P}_{11}$ & $\xrightarrow{ \big(te_1+t^2e_2,\  te_2,\  t^2e_3,\  t^2e_4 \big)}$ & $\mathfrak{P}_{02}$ 
 \\  \hline

$\mathfrak{P}_{04}^{0}$  & $\xrightarrow{ \big(te_1,\ e_2,\ te_3,\ te_4\big)}$ & $\mathfrak{P}_{03}$
\\  \hline

$\mathfrak{P}_{16}^{t^{-1}}$   & 
$\xrightarrow{ 
\big(\sqrt[3]{\frac{t^2-t^3}{\alpha}}e_1-\frac{1}{\sqrt[3]{\alpha^2 (t^2-t^3)}}e_2+\frac{\alpha t^2-1}{2\alpha t^3}e_3, \ 
\sqrt[3]{\frac{\alpha t}{1-t}}e_2+\frac{2}{t^3-t^2}e_3, \ 
\sqrt[3]{\frac{1-t}{\alpha t^4}}e_4,\  
e_3+\frac{t-3}{\sqrt[3]{\alpha t^7 (1-t)^2}}e_4 \big)}$  & $\mathfrak{P}_{04}^{\alpha}$ 
\\  \hline

$\mathfrak{P}_{16}^{t^{-1}}$   & $\xrightarrow{ \big((t^{2}-t)e_1+(t-1)e_2,\  
(t^2-t)^2e_2+2(t-1)^2e_3-\frac{(t-1)^3}{t^2}e_4,\  
t^{2}(t-1)^3e_3-(t-1)^3(t-3)e_4,\  t(t-1)^4e_4 \big)}$    & $\mathfrak{P}_{05}$ 
\\  \hline

$\mathfrak{P}_{16}^{-t}$  & $\xrightarrow{ \big(te_1,\  te_2-t^2e_4,\  t^{2}e_3,\  t^{3}e_4 \big)}$ & $\mathfrak{P}_{06}$ \\  \hline
$\mathfrak{P}_{06}$  & $\xrightarrow{ \big(t^{-1}e_1,\  t^{-2}e_2,\ t^{-3}e_3,\  t^{-4}e_4\big)}$ & $\mathfrak{P}_{07}$
\\  \hline

$\mathfrak{P}_{06}$  & $\xrightarrow{ \big(t^{-1}e_1,\  e_2,\  t^{-1}e_3,\  t^{-2}e_4\big)}$ & $\mathfrak{P}_{08}$  \\  \hline

$\mathfrak{P}_{06}$  & $\xrightarrow{ \big(\mathrm{i}  e_1-e_2,\  -te_2,\  -\mathrm{i} te_3,\  te_4\big)}$ & $\mathfrak{P}_{09}$
\\  \hline

$\mathfrak{P}_{11}$  & $\xrightarrow{ \big(t^{-1}e_1,\ t^{-1}e_2,\ t^{-2}e_3,\ t^{-3}e_4\big)}$ & $\mathfrak{P}_{10}$  
\\  \hline

$\mathfrak{P}_{16}^{t}$  & $\xrightarrow{ \big(\frac{t}{1-t}e_2+\frac{t}{(1-t)^2}e_3+\frac{t}{(1-t)^3}e_4,\  
te_1,\  
\frac{t^2}{t-1}e_3-\frac{t^2}{(1-t)^2}e_4,\  
\frac{t^3}{t-1}e_4 \big)}$   & $\mathfrak{P}_{11}$ 
\\  \hline

$\mathfrak{P}_{16}^{t}$  & $\xrightarrow{ \big(\frac{1}{\sqrt{t}}e_2+\frac{t-2\sqrt{t}-1}{t^2-t}e_3+\frac{\sqrt{t}-2}{t^2-t}e_4, \  
e_1-\frac{1}{t^2-t}e_2,\  
-\frac{1}{\sqrt{t}}e_3+\frac{1+2\sqrt{t}-t}{t^2-t}e_4,\  
-\frac{1}{\sqrt{t}}e_4\big)}$  & $\mathfrak{P}_{12}$ 
\\  \hline

$\mathfrak{P}_{15}^{0}$  & $\xrightarrow{ \big(t^{-1}e_1,\  t^{-1}e_2,\  t^{-2}e_3,\ t^{-3}e_4 \big)}$ & $\mathfrak{P}_{13}$ \\  \hline

$\mathfrak{P}_{15}^{t^{-1}}$  & $\xrightarrow{ \big(t^{-1}e_1,\  t^{-1}e_2,\  t^{-2}e_3,\  t^{-3}e_4 \big)}$ & $\mathfrak{P}_{14}$

\\  \hline

$\mathfrak{P}_{16}^{t}$   & $\xrightarrow{ \big((t-1)e_1+(t-1)t^{-1}e_2,\  
(t-1)^2e_2+2(t-1)^2e_3+(t-1)^3t^{-1}e_4, \ 
(t-1)^3e_3+2(t-1)^3e_4,\  
(t-1)^4e_4 \big)}$  & $\mathfrak{P}_{17}$ 
\\  \hline

$\mathfrak{P}_{16}^{1+t}$    & $\xrightarrow{ \big((t+1)e_1+(t^{-1}+1)e_2,\  
(t+1)^2e_2+\frac{2(t+1)^3}{t}e_3+\frac{(t+1)^3}{t}e_4,\ 
(t+1)^3e_3+\frac{2(t+1)^4}{t}e_4,\  
(t+1)^4e_4 \big)}$  & $\mathfrak{P}_{18}$ 
\\  \hline

$\mathfrak{P}_{11}$  & $\xrightarrow{ \big(e_1,\ te_2,\ t^{-1}e_3,\ e_4 \big)}$ & $\mathfrak{P}_{19}$ \\  \hline

$\mathfrak{P}_{16}^{t}$  & $\xrightarrow{ \big(te_1,\  e_2+(t^{-1}-1)e_4,\  (t-1)e_3,\  t(t-1)e_4 \big)}$ & $\mathfrak{P}_{20}$  
 \\  \hline

$\mathfrak{P}_{16}^{t}$   & $\xrightarrow{ \big(te_1-\frac{1}{1-t}e_2,\ 
te_1-\frac{1+t-t^2}{1-t}e_2+\frac{2t}{1-t}e_3+\frac{1+t-t^2}{1-t}e_4,\ 
te_3-\frac{2t}{1-t}e_4,\ 
t^2e_4 \big)}$    & $\mathfrak{P}_{21}$ 
\\  \hline

$\mathfrak{P}_{15}^{-\frac{\alpha}{2}}$  & $\xrightarrow{ 
\big(\frac{t}{2}e_3-\frac{\alpha t }{4}e_4, \ 2te_2,\  -te_1-2t^3e_2, \  t^2e_4 \big)}$    & $\mathfrak{P}_{22}^{\alpha}$ 
\\  \hline

$\mathfrak{P}_{25}^{\frac{1}{t-t^2}}$  & $\xrightarrow{ \big(te_1-\frac{1}{1-t}e_3,\ 
t^2e_2+t^3e_3,\ 
te_2+te_3-\frac{1}{t^2-t^3}e_4,\ 
t^2e_4 \big)}$    & $\mathfrak{P}_{23}$ 
\\  \hline

$\mathfrak{P}_{16}^{t}$   & $\xrightarrow{ \big((t-1)e_1+\frac{t-1}{t^{2}}e_2,\  
(t-1)^2e_2+\frac{2(t-1)^2}{t}e_3+(1-t^{-1})^3e_4,\  
\frac{(t-1)^3}{t}e_3+\frac{2(t-1)^3}{t^2}e_4,\ 
(t-1)^4t^{-1}e_4 \big)}$   & $\mathfrak{P}_{24}$ 
\\  \hline

$\mathfrak{P}_{15}^{\frac{1}{4\alpha^2}}$  & 
$\xrightarrow{ \big(\frac{t^2}{8\alpha^2}e_1+\frac{t^2}{4\alpha}e_2-\frac{t^2}{16\alpha^3}e_3,\  
\frac{t^4}{16\alpha^2}e_3,\  \frac{t^3}{4\alpha}e_2,\  \frac{t^6}{64\alpha^4}e_4\big)}$   & $\mathfrak{P}_{25}^{\alpha\neq 0}$ 
\\  \hline

$\mathfrak{P}_{16}^{t^{-1}}$ & $\xrightarrow{ \big(te_1-\frac{1-\alpha t}{1-t}e_2,\  
t^2e_2-\frac{2-2\alpha t}{1-t}e_3+\frac{(1-\alpha t)^2}{t^2-t^3}e_4,\ 
te_3-\frac{3-(1+2\alpha) t}{t-t^2}e_4,\  te_4)}$    & $\mathfrak{P}_{26}^{\alpha}$ 

\\  \hline

$\mathfrak{P}_{16}^{t^{-1}}$  & $\xrightarrow{ \big(te_1-\frac{t}{1-t}e_2,\ 
t^2e_2-\frac{2t}{1-t}e_3+\frac{1}{1-t}e_4,\ 
t^2e_3-\frac{3t-t^2}{1-t}e_4,\  
t^2e_4 \big)}$    & $\mathfrak{P}_{27}$ 

\\  \hline

\end{longtable}

\noindent
By direct calculations,  we have 
$\dim  \mathcal{O}(\mathfrak{P}_{15}^{\alpha})\ =\  \dim  \mathcal{O}(\mathfrak{P}_{16}^{\alpha}) \ =\ 14.$

\end{proof}

\end{document}